\documentclass[11pt]{article}
\usepackage{amsmath}
\usepackage{amssymb}
\usepackage{amsthm}
\usepackage[usenames]{color}
\usepackage{amscd}
\usepackage{dsfont}
\usepackage{indentfirst}

\usepackage[colorlinks=true,linkcolor=blue,filecolor=red,
citecolor=webgreen]{hyperref}
\definecolor{webgreen}{rgb}{0,.5,0}

\newcommand{\seqnum}[1]{\href{http://oeis.org/#1}{\underline{#1}}}

\def\R{{\mathbb{R}}}
\def\C{{\mathbb{C}}}
\def\P{{\mathbb{P}}}
\def\RE{\operatorname{Re}}
\def\pow{\operatorname{pow}}
\def\altern{\operatorname{altern}}

\theoremstyle{plain}
\newtheorem{theorem}{Theorem}
\newtheorem{corollary}[theorem]{Corollary}
\newtheorem{lemma}[theorem]{Lemma}
\newtheorem{proposition}[theorem]{Proposition}

\theoremstyle{definition}

\newtheorem{application}[theorem]{Application}
\newtheorem{open}[theorem]{Open Problem}

\theoremstyle{remark}
\newtheorem{remark}[theorem]{Remark}

\begin{document}

\title{\bf Alternating sums concerning multiplicative arithmetic functions
\thanks{The present scientific contribution is dedicated to
the 650th anniversary of the foundation of the University of P\'ecs,
Hungary.}}
\author{L\'aszl\'o T\'oth \\
Department of Mathematics, University of P\'ecs \\ Ifj\'us\'ag
\'utja 6, H-7624 P\'ecs, Hungary \\ E-mail: {\tt
ltoth@gamma.ttk.pte.hu}}
\date{}
\maketitle

\centerline{Journal of Integer Sequences {\bf 20} (2017), Article 17.2.1}
\medskip

\abstract{We deduce asymptotic formulas for the alternating sums $\sum_{n\le x} (-1)^{n-1} f(n)$
and \newline $\sum_{n\le x} (-1)^{n-1} \frac1{f(n)}$, where $f$ is one of the following classical
multiplicative arithmetic functions: Euler's totient function, the Dedekind function, the sum-of-divisors
function, the divisor function, the gcd-sum function. We also consider analogs of these functions, which are
associated to unitary and exponential divisors, and other special functions. Some of our results improve the error
terms obtained by Bordell\`{e}s and Cloitre. We formulate certain open problems.}

\medskip



\tableofcontents

\section{Introduction}

Alternating sums and series appear in various topics of mathematics and number theory, in particular.
For example, it is well-known that for $s\in \C$ with $\Re s >1$,
\begin{equation} \label{eta}
\eta(s):=\sum_{n=1}^{\infty} (-1)^{n-1} \frac1{n^s} =
\left(1-\frac1{2^{s-1}} \right) \zeta(s),
\end{equation}
representing the alternating zeta function or Dirichlet's eta
function. Here the left-hand side is convergent for $\Re
s>0$, and this can be used for analytic continuation
of the Riemann zeta function for $\Re s >0$.
See, e.g., Tenenbaum \cite[Sect.\ II.3.2]{Ten2015}.

Bordell\`{e}s and Cloitre \cite{BorClo2013} established asymptotic
formulas with error terms for alternating sums
\begin{equation} \label{altern_sum_f}
\sum_{n\le x} (-1)^{n-1} f(n),
\end{equation}
where $f(n)=1/g(n)$ and $g$ belongs to a class of multiplicative arithmetic functions, including Euler's totient
function $\varphi$, the sum-of-divisors function $\sigma$ and the Dedekind function $\psi$. It seems that there are no other
results in the literature for alternating sums of type \eqref{altern_sum_f}.

Using a different approach, also based on the convolution method, we show that for many classical multiplicative
arithmetic functions $f$, estimates with sharp error terms for the alternating sum \eqref{altern_sum_f} can easily be deduced by
using known results for
\begin{equation}  \label{sum_f}
\sum_{n\le x} f(n).
\end{equation}

For other given multiplicative functions $f$, a difficulty arises,
namely to estimate the coefficients of the reciprocal of a formal
power series, more exactly the reciprocal of the Bell series of $f$
for $p=2$. If the coefficients of the original power series are
positive and log-convex, then a result of Kaluza \cite{Kal1928} can
be used. The obtained error terms for \eqref{altern_sum_f} are
usually the same, or slightly larger than for \eqref{sum_f}.

In this way we improve some of the error terms obtained in \cite{BorClo2013}.
We also deduce estimates for other classical
multiplicative functions $f$. As a tool, we use formulas for
alternating Dirichlet series
\begin{equation} \label{altern_Dir}
D_{\altern}(f,s):=\sum_{n=1}^{\infty} (-1)^{n-1} \frac{f(n)}{n^s},
\end{equation}
generalizing \eqref{eta}.

In the case of some other functions $f$, a version of Kendall's renewal theorem (from probability theory) can be applied. 
Berenhaut, Allen, and Fraser obtained \cite{Ber2006}
an explicit form of Kendall's theorem (also see \cite{Ber2008}),
but this cannot be used for the functions we deal with.
We prove a new explicit Kendall-type inequality, which can be applied in some cases. As far as we know, there are no other
similar applicable results to obtain better error terms in the literature.
We formulate several open problems concerning the error terms of the
presented asymptotic formulas.

Finally, a generalization of the alternating Dirichlet series \eqref{altern_Dir} and the alternating sum \eqref{altern_sum_f} is discussed.

\section{General results}

\subsection{Alternating Dirichlet series} \label{Sect_Alter_Dir_series}

Let
\begin{equation} \label{Dir}
D(f,s):=\sum_{n=1}^{\infty} \frac{f(n)}{n^s}
\end{equation}
denote the Dirichlet series of the function $f$. If $f$ is multiplicative, then it can be expanded into the Euler product
\begin{equation} \label{multipl_Euler}
D(f,s)=\prod_{p\in \P} \sum_{\nu=0}^{\infty}
\frac{f(p^{\nu})}{p^{\nu s}}.
\end{equation}

If $f$ is completely multiplicative, then
\begin{equation} \label{compl_mult_sum}
\sum_{\nu=0}^{\infty} \frac{f(p^{\nu})}{p^{\nu s}} =
\left(1-\frac{f(p)}{p^s}\right)^{-1}
\end{equation}
and
\begin{equation} \label{compl_multipl_Euler}
D(f,s)=\prod_{p\in \P} \left(1-\frac{f(p)}{p^s}\right)^{-1}.
\end{equation}

\begin{proposition} \label{prop_1}
If $f$ is a multiplicative function, then
\begin{equation} \label{form_Dir_altern}
\sum_{n=1}^{\infty} (-1)^{n-1} \frac{f(n)}{n^s} = D(f,s) \left(
2\left( \sum_{\nu=0}^{\infty}\frac{f(2^\nu)}{2^{\nu s}} \right)^{-1}
-1\right),
\end{equation}
and if $f$ is completely multiplicative, then
\begin{equation*}
\sum_{n=1}^{\infty} (-1)^{n-1} \frac{f(n)}{n^s} =
\left(1-\frac{f(2)}{2^{s-1}} \right) \prod_{p\in \P} \left(1-\frac{f(p)}{p^s} \right)^{-1},
\end{equation*}
formally or in case of convergence.
\end{proposition}

\begin{proof}
We have by using \eqref{multipl_Euler},
\begin{gather*}
\sum_{n=1}^{\infty} (-1)^{n-1} \frac{f(n)}{n^s} = - \sum_{n=1}^{\infty}
\frac{f(n)}{n^s} +  2 \sum_{\substack{n=1\\ n \text{ odd}}}^{\infty}
\frac{f(n)}{n^s} = - D(f,s) +2 \prod_{\substack{p\in \P\\ p>2}}
\sum_{\nu=0}^{\infty} \frac{f(p^{\nu})}{p^{\nu s}} \\ = - D(f,s) +2
D(f,s) \left(\sum_{\nu=0}^{\infty} \frac{f(2^{\nu})}{2^{\nu s}}\right)^{-1}
 = D(f,s) \left( 2 \left( \sum_{\nu=0}^{\infty}\frac{f(2^\nu)}{2^{\nu
s}} \right)^{-1} -1\right).
\end{gather*}

If $f$ is completely multiplicative, then use identities
\eqref{compl_mult_sum} and \eqref{compl_multipl_Euler}.
\end{proof}

For special choices of $f$ we obtain formulas for the alternating Dirichlet series \eqref{altern_Dir}.
For example, let $f=\varphi$ be Euler's totient function. For every prime $p$,
\begin{equation} \label{varphi_p}
\sum_{\nu=0}^{\infty} \frac{\varphi(p^\nu)}{p^{\nu s}} =
\left(1-\frac1{p^s}\right) \left(1-\frac1{p^{s-1}}\right)^{-1}.
\end{equation}

Here the left-hand side of \eqref{varphi_p} can be computed
directly. However, it is more convenient to use the well-known representation of the Dirichlet series of $\varphi$ (similar
considerations are valid for other classical multiplicative function, as well). Namely,
\begin{equation} \label{111}
\sum_{n=1}^{\infty} \frac{\varphi(n)}{n^s} =
\frac{\zeta(s-1)}{\zeta(s)} = \prod_{p\in \P}
\left(1-\frac1{p^s}\right) \left(1-\frac1{p^{s-1}}\right)^{-1},
\end{equation}
and using the Euler product,
\begin{equation} \label{222}
\sum_{n=1}^{\infty} \frac{\varphi(n)}{n^s} = \prod_{p\in \P}
\sum_{\nu=0}^{\infty} \frac{\varphi(p^\nu)}{p^{\nu s}}.
\end{equation}

Now a quick look at \eqref{111} and \eqref{222} gives \eqref{varphi_p}.
We deduce from Proposition \ref{prop_1} that
\begin{equation} \label{D_altern_varphi_first}
D_{\altern}(\varphi,s) = \frac{\zeta(s-1)}{\zeta(s)} \left( 2
\left(1-\frac1{2^s}\right)^{-1} \left(1-\frac1{2^{s-1}} \right)
-1\right),
\end{equation}
which can be written as \eqref{D_altern_varphi}.

Note that the function $n\mapsto (-1)^{n-1}$ is multiplicative. Therefore, it is
possible to give a direct proof of \eqref{D_altern_varphi_first}
(and of similar formulas, where $\varphi$ is replaced by another multiplicative
function) using Euler products:
\begin{equation*}
\sum_{n=1}^{\infty} \frac{(-1)^{n-1}\varphi(n)}{n^s} =
\left(1-\sum_{\nu=1}^{\infty} \frac{\varphi(2^{\nu})}{2^{\nu s}}
\right) \prod_{\substack{p\in \P\\ p>2}} \left(1 + \sum_{\nu=1}^{\infty}
\frac{\varphi(p^{\nu})}{p^{\nu s}} \right),
\end{equation*}
but computations are simpler by the previous approach.

\subsection{Mean values and alternating sums}

Let $f$ be a complex-valued arithmetic function. The (asymptotic) mean value of $f$ is
\begin{equation*}
M(f):=\lim_{x\to \infty} \frac1{x} \sum_{n\le x} f(n),
\end{equation*}
provided that this limit exists. Let
\begin{equation*}
M_{\altern}(f):=\lim_{x\to \infty} \frac1{x} \sum_{n\le x} (-1)^{n-1} f(n)
\end{equation*}
denote the mean value of the function $n\mapsto (-1)^{n-1} f(n)$ (if it exists).

\begin{proposition} \label{prop_mean} Assume that $f$ is a multiplicative function  and
\begin{equation} \label{cond}
\sum_{p\in \P} \frac{|f(p)-1|}{p} < \infty, \qquad \sum_{p\in \P} \sum_{\nu=2}^{\infty} \frac{|f(p^\nu)|}{p^\nu} < \infty.
\end{equation}

Then there exists
\begin{equation*}
M(f) = \prod_{p\in \P} \left(1-\frac1{p}\right) \sum_{\nu=0}^{\infty}
\frac{f(p^\nu)}{p^\nu}.
\end{equation*}

Furthermore, if $\sum_{\nu=0}^{\infty} \frac{f(2^\nu)}{2^\nu} \ne 0$, then there exists
\begin{equation} \label{nonzero}
M_{\altern}(f) = M(f) \left(2\left( \sum_{\nu=0}^{\infty}
\frac{f(2^\nu)}{2^\nu}\right)^{-1}-1\right),
\end{equation}
and if $\sum_{\nu=0}^{\infty} \frac{f(2^\nu)}{2^\nu}=0$, then $M(f)=0$ and
there exists
\begin{equation} \label{zero}
M_{\altern}(f) = \prod_{\substack{p\in \P\\p>2}}
\left(1-\frac1{p}\right) \sum_{\nu=0}^{\infty}
\frac{f(p^\nu)}{p^\nu}.
\end{equation}
\end{proposition}

\begin{proof} The result for $M(f)$ is a version of Wintner's theorem for multiplicative functions. See Schwarz
and Spilker \cite[Cor.\ 2.3]{SchSpi1994}. It easy to check that
assuming \eqref{cond} for $f$, the same conditions hold for the
multiplicative function $n\mapsto (-1)^{n-1} f(n)$. We deduce that
$M_{\altern}(f)$ exists and it is
\begin{equation} \label{M_altern}
M_{\altern}(f) = \left(1-\frac1{2}\right)\left(1- \sum_{\nu=1}^{\infty} \frac{f(2^{\nu})}{2^{\nu}} \right)  \prod_{\substack{p\in \P\\ p>2}}
\left(1-\frac1{p}\right) \sum_{\nu=0}^{\infty} \frac{f(p^{\nu})}{p^{\nu}}.
\end{equation}

Now if $t:=\sum_{\nu=1}^{\infty} \frac{f(2^\nu)}{2^\nu}\ne -1$, then
\begin{equation*}
M_{\altern}(f)= \frac{1-t}{1+t} \prod_{p\in \P} \left(1-\frac1{p}\right) \sum_{\nu=0}^{\infty} \frac{f(p^{\nu})}{p^{\nu}},
\end{equation*}
which is \eqref{nonzero}. If $t=-1$, then \eqref{M_altern} gives \eqref{zero}.
\end{proof}

\begin{application}  Let $f$ be multiplicative such that $f(p)=1$, $f(p^2)=-6$,
$f(p^{\nu})=0$ for every $p\in \P$ and every $\nu \ge 3$. Here
$\sum_{\nu=0}^{\infty} \frac{f(2^\nu)}{2^\nu} =1+\frac1{2}-\frac{6}{4}=0$. Using
Proposition \ref{prop_mean} we deduce that $M(f)=0$ and there exists
\begin{equation*}
\lim_{x\to \infty} \frac1{x} \sum_{n\le x} (-1)^{n-1} f(n)= \prod_{\substack{p\in \P \\ p>2}}
\left(1-\frac1{p}\right) \left(1+\frac1{p}-\frac{6}{p^2} \right) = \prod_{\substack{p\in \P \\ p>2}}
\left(1-\frac{7}{p^2}+\frac{6}{p^3} \right) \ne 0.
\end{equation*} 
\end{application}

The following result is similar. Let
\begin{equation*}
L(f):=\lim_{x\to \infty} \frac1{\log x} \sum_{n\le x} \frac{f(n)}{n}
\end{equation*}
denote the logarithmic mean value of $f$ and, assuming that $f$ is non-vanishing, let
\begin{equation*}
\overline{L}(f):= \lim_{x\to \infty} \frac1{\log x} \sum_{n\le x} \frac1{f(n)},
\end{equation*}
\begin{equation*}
\overline{L}_{\altern}(f):= \lim_{x\to \infty} \frac1{\log x} \sum_{n\le x} (-1)^{n-1} \frac1{f(n)},
\end{equation*}
provided that the limits exist.

\begin{proposition} \label{prop_mean_2} Assume that $f$ is a non-vanishing multiplicative function  and
\begin{equation} \label{cond_2}
\sum_{p\in \P} \left| \frac1{f(p)}-\frac1{p}\right| < \infty, \qquad \sum_{p\in \P} \sum_{\nu=2}^{\infty} \frac1{|f(p^\nu)|} < \infty.
\end{equation}

Then there exists
\begin{equation*}
\overline{L}(f) = \prod_{p\in \P} \left(1-\frac1{p}\right) \sum_{\nu=0}^{\infty}
\frac1{f(p^\nu)}.
\end{equation*}

Furthermore, if $\sum_{\nu=0}^{\infty} \frac1{f(2^\nu)}\ne 0$, then there exists
\begin{equation*}
\overline{L}_{\altern}(f) = \overline{L}(f) \left(2\left( \sum_{\nu=0}^{\infty}
\frac1{f(2^\nu)}\right)^{-1}-1\right),
\end{equation*}
and if $\sum_{\nu=0}^{\infty} \frac1{f(2^\nu)}=0$, then $\overline{L}(f)=0$ and
there exists
\begin{equation*}
\overline{L}_{\altern}(f) = \prod_{\substack{p\in \P \\ p>2}}
\left(1-\frac1{p}\right) \sum_{\nu=0}^{\infty}
\frac1{f(p^\nu)}.
\end{equation*}
\end{proposition}

\begin{proof} Apply Proposition \ref{prop_mean} for $f(n):=\frac{n}{f(n)}$ and use the following property: If the mean value
$M(f)$ exists, then the logarithmic mean value $L(f)$ exists as
well, and is equal to $M(f)$. See Hildebrand \cite[Thm.\
2.13]{Hil2013}.
\end{proof}

\begin{application} It follows from Proposition \ref{prop_mean_2} that
\begin{equation*}
\overline{L}(\varphi) = \lim_{x\to \infty} \frac1{\log x} \sum_{n\le x} \frac1{\varphi(n)} =
\prod_{p\in \P} \left(1-\frac1{p} \right) \sum_{\nu=0}^{\infty} \frac1{\varphi(p^\nu)}= \frac{\zeta(2)\zeta(3)}{\zeta(6)},
\end{equation*}
which is well-known, and
\begin{equation*}
\lim_{x\to \infty} \frac1{\log x} \sum_{n\le x} (-1)^{n-1} \frac1{\varphi(n)} = \overline{L}(\varphi) \left(2\left( \sum_{\nu=0}^{\infty}
\frac1{\varphi(2^\nu)}\right)^{-1}-1\right)= -\frac{\zeta(2)\zeta(3)}{3\zeta(6)},
\end{equation*}
obtained by Bordell\`{e}s and Cloitre \cite{BorClo2013}. Conditions
\eqref{cond_2} were refined in \cite{BorClo2013} to deduce
asymptotic formulas with error terms for alternating sums of
reciprocals of a class of multiplicative arithmetic functions,
including Euler's totient function.
\end{application}

\subsection{Method to obtain asymptotic formulas} \label{Subsect_Method}

Assume that $f$ is a nonzero complex-valued multiplicative function.
Consider the formal power series
\begin{equation*}
S_f(x):=\sum_{\nu=0}^{\infty} a_\nu x^\nu,
\end{equation*}
where $a_\nu=f(2^\nu)$ ($\nu \ge 0$), $a_0=f(1)=1$. Note that
$S_f(x)$ is the Bell series of the function $f$ for the prime $p=2$.
See, e.g., Apostol \cite[Ch.\ 2]{Apo1976}. Let
\begin{equation*}
\overline{S}_f(x): =\sum_{\nu=0}^{\infty} b_\nu x^\nu
\end{equation*}
be its formal reciprocal power series. Here the coefficients $b_\nu$
are given by $b_0=1$ and $\sum_{j=0}^\nu a_j b_{\nu-j}=0$ ($\nu\ge
1$). If both series $S_f(x)$ and $\overline{S}_f(x)$ converge for an $x\in \C$, then $S_f(x) \overline{S}_f(x)
=1$. In particular, if $r_f$ and $\overline{r}_f$ are the radii of convergence of
$S_f(x)$, respectively $\overline{S}_f(x)$, then $S_f(x) \overline{S}_f(x)
=1$ for every $x\in \C$ such that $|x|<\min (r_f,\overline{r}_f)$.

It follows from \eqref{form_Dir_altern} that the convolution identity
\begin{equation} \label{convo}
(-1)^{n-1} f(n) = \sum_{dj =n} h_f(d) f(j) \quad (n\ge 1)
\end{equation}
holds, where the function $h_f$ is multiplicative, $h_f(p^\nu)=0$ if
$p>2$, $\nu \ge 1$ and $h_f(2^\nu)=2b_{\nu}$ ($\nu\ge 1$),
$h_f(1)=2b_0-1=1$.

Therefore, by the convolution method,
\begin{equation} \label{sum}
\sum_{n\le x} (-1)^{n-1} f(n) = \sum_{d\le x} h_f(d) \sum_{j \le x/d} f(j),
\end{equation}
which leads to a good estimate for \eqref{altern_sum_f} if an asymptotic
formula for $\sum_{n\le x} f(n)$ is known and if the coefficients
$b_\nu$ of above can be well estimated. Note that, according to
\eqref{convo} and \eqref{form_Dir_altern},
\begin{equation} \label{const}
\sum_{n=1}^{\infty} \frac{h_f(n)}{n^s} = \frac{2}{S_f(1/2^s)}-1,
\end{equation}
provided that both $S_f(1/2^s)$ and  $\overline{S}_f(1/2^s)$ converge. By differentiating,
\begin{equation} \label{const_diff}
\sum_{n=1}^{\infty} \frac{h_f(n)\log n}{n^s} = -\frac{\log 2}{2^{s-1}}\cdot \frac{S_f'(1/2^s)}{S_f(1/2^s)^2},
\end{equation}
assuming that $|1/2^s|<\min (r_f,\overline{r}_f)$. Identities \eqref{const} and \eqref{const_diff} will be used in
applications.

\subsection{Two general asymptotic formulas}

We prove two general results that will be applied for several
special functions in Section \ref{Section_3}.

\begin{proposition} \label{Prop_f} Let $f$ be a multiplicative function. Assume that

(i) there exists a constant $C_f$ such that
\begin{equation*}
\sum_{n\le x} f(n) = C_f x^2 + O\left( x R_f(x) \right),
\end{equation*}
where $1\ll R_f(x)=o(x)$ as $x\to \infty$, and $R_f(x)$ is nondecreasing;

(ii) $S_f(1/4)$ converges;

(iii) the sequence $(b_\nu)_{\nu \ge 0}$ of coefficients of the reciprocal power series
$\overline{S}_f(x)$ is bounded.

Then
\begin{equation*}
\sum_{n \le x} (-1)^{n-1} f(n) =  C_f \left(\frac2{S_f(1/4)} -1
\right) x^2 + O\left(x R_f(x) \right).
\end{equation*}
\end{proposition}

\begin{proof} According to \eqref{sum},
\begin{equation*}
\sum_{n\le x} (-1)^{n-1} f(n) = \sum_{d\le x} h_f(d) \left(C_f
\frac{x^2}{d^2} + O\left(\frac{x}{d} R_f(x/d) \right) \right)
\end{equation*}
\begin{equation*}
= C_f x^2 \sum_{d\le x} \frac{h_f(d)}{d^2}+  O \left(x R_f(x)
\sum_{d\le x} \frac{|h_f(d)|}{d} \right).
\end{equation*}

Since the sequence $(b_\nu)_{\nu \ge 0}$ is bounded, the function
$h_f$ is bounded. Moreover, the sum
\begin{equation*}
\sum_{d\le x} \frac{|h_f(d)|}{d} = \sum_{d=2^{\nu} \le x}
\frac{|h_f(2^{\nu})|}{2^{\nu}} \ll \sum_{2^{\nu}\le x} \frac{|b_{\nu}|}{2^{\nu}}
\end{equation*}
is bounded, as well. Note that $S_f(1/4)$ and $\overline{S}_f(1/4)$ both converge by conditions (ii) and (iii).
We deduce, by using \eqref{const} for $s=2$, that
\begin{equation*}
\sum_{n\le x} (-1)^{n-1} f(n) = C_f x^2 \sum_{d=1}^{\infty}
\frac{h_f(d)}{d^2} + O\left( x^2 \sum_{d>x} \frac1{d^2}\right) + O
\left(x R_f(x) \right)
\end{equation*}
\begin{equation*}
=C_f x^2 \left(\frac{2}{S_f(1/4)}-1 \right) + O \left(x R_f(x) \right).
\end{equation*}
\end{proof}

\begin{proposition} \label{Prop_1_per_f} Let $f$ be a nonvanishing multiplicative function. Assume that

(i) there exist constants $D_f$ and $E_f$ such that
\begin{equation} \label{asympt_D_f_E_f}
\sum_{n\le x} \frac1{f(n)} = D_f (\log x + E_f) + O\left( x^{-1} R_{1/f}(x) \right),
\end{equation}
where $1\ll R_{1/f}(x)=o(x)$ as $x\to \infty$, and $R_{1/f}(x)$ is nondecreasing;

(ii) the radius of convergence of the series $S_{1/f}(x)$ is $r_{1/f}>1$;

(iii) the coefficients $b_\nu$ of the reciprocal power series $\overline{S}_{1/f}(x)$ satisfy $b_{\nu} \ll M^\nu$ as $\nu \to \infty$,
where $0<M<1$ is a real number.

Then
\begin{equation} \label{asympt_altern_1_per_f}
\sum_{n \le x} (-1)^{n-1} \frac1{f(n)} = D_f \left( \left(\frac{2}{S_{1/f}(1)}-1\right) (\log x + E_f)
+ 2(\log 2)\frac{S'_{1/f}(1)}{S_{1/f}(1)^2}\right) + O\left(T_{1/f}(x) \right),
\end{equation}
where
\begin{equation} \label{error_M}
T_{1/f}(x) = \begin{cases} x^{-1} R_{1/f}(x), & \ \text{ if } 0<M<\frac1{2}; \\
x^{-1} R_{1/f}(x) \log x, & \ \text{ if } M=\frac1{2}; \\
x^{\log M/\log 2} \max(\log x, R_{1/f}(x)), & \ \text{ if } \frac1{2}<M<1.
\end{cases}
\end{equation}
\end{proposition}

\begin{proof} According to \eqref{sum} we deduce that
\begin{align*}
\sum_{n\le x} (-1)^{n-1} \frac1{f(n)} 
&= \sum_{d\le x} h_{1/f}(d) \sum_{j\le x/d} \frac1{f(j)} \\
&= \sum_{d\le x} h_{1/f}(d) \left( D_f \left( \log
\frac{x}{d} + E_f \right) + O\left( (x/d)^{-1} R_{1/f}(x/d) \right) \right) \\
&= D_f (\log x + E_f) \sum_{d\le x} h_{1/f}(d) - D_f
\sum_{d\le x} h_{1/f}(d) \log d \\
&+ O \left(x^{-1} R_{1/f}(x) \sum_{d\le x} d |h_{1/f}(d)| \right).
\end{align*}

That is,
\begin{equation*}
\sum_{n\le x} (-1)^{n-1} \frac1{f(n)} = D_f (\log x+ E_f)
\sum_{d=1}^{\infty} h_{1/f}(d) + O\left(\log x \sum_{d>x}
|h_{1/f}(d)| \right)
\end{equation*}
\begin{equation} \label{summ}
 - D_f \sum_{d=1}^{\infty} h_{1/f}(d)
\log d  + O\left( \sum_{d>x} |h_{1/f}(d)| \log d \right)
+ O \left(x^{-1} R_{1/f}(x) \sum_{d\le x} d |h_{1/f}(d)|\right).
\end{equation}

Note that $\min(r_{1/f},\overline{r}_{1/f})>1$ by conditions (ii) and (iii). By using \eqref{const} and \eqref{const_diff} for $s=0$,
\begin{gather*}
\sum_{d=1}^{\infty} h_{1/f}(d) = \frac{2}{S_{1/f}(1)}-1, \\
\sum_{d=1}^{\infty} h_{1/f}(d)\log d = -2(\log 2) \frac{S'_{1/f}(1)}{S_{1/f}(1)^2}.
\end{gather*}

Furthermore,
\begin{gather*}
\sum_{d>x} |h_{1/f}(d)| = \sum_{d=2^{\nu}>x} |h_{1/f}(2^{\nu})|
\ll \sum_{2^{\nu}>x} |b_{\nu}| \ll \sum_{2^{\nu}>x} M^{\nu} \ll x^{\log M/\log 2}, \\
\sum_{d>x} |h_{1/f}(d)| \log d \ll \sum_{2^{\nu}>x}  \nu |b_{\nu}| \ll \sum_{2^{\nu}>x} \nu M^{\nu}
\ll x^{\log M/\log 2} \log x, \\
\sum_{d\le x} d |h_{1/f}(d)| = \sum_{2^{\nu}\le x} 2^{\nu} |b_{\nu}|
\ll  \sum_{\nu \le \log x/\log 2} (2M)^{\nu},
\end{gather*}
where the latter sum is bounded if $0<M<1/2$, it is $\ll \log x$ if $M=1/2$, and is $\ll x^{1+ \log M/\log 2}$ if $1/2<M<1$.

Inserting these into \eqref{summ}, the proof is complete.
\end{proof}

\section{Estimates on coefficients of reciprocal power series}
\label{Estimates}

As mentioned in Section \ref{Subsect_Method}, in order to deduce sharp error terms for
alternating sums \eqref{altern_sum_f} we need good estimates for the
coefficients $b_{\nu}$ of the power series $\overline{S}_f(x)$.

\subsection{Theorem of Kaluza}

In many (nontrivial) cases the next result can be used.

\begin{lemma} \label{lemma_Kaluza} Let $\sum_{\nu =0}^{\infty} a_\nu x^\nu$
be a power series such that $a_\nu >0$ {\rm (}$\nu \ge 0${\rm )}
and the sequence $(a_\nu)_{\nu \ge 0}$ is log-convex, that is
$a_{\nu}^2 \le a_{\nu-1} a_{\nu+1}$ {\rm (}$\nu \ge 1${\rm )}. Then
for the coefficients $b_\nu$ of the (formal) reciprocal power series
$\sum_{\nu =0}^{\infty} b_\nu x^\nu$ one has  $b_0 = 1/a_0 > 0$ and
\begin{equation*}
- \frac1{a_0^2} a_{\nu}\le b_\nu \le 0 \quad \text{ for all $\nu \ge 1$}.
\end{equation*}
\end{lemma}

\begin{proof} The property that $b_\nu \le 0 $ for all $\nu \ge 1$ is the theorem of Kaluza \cite[Satz
3]{Kal1928}. See \cite{Car1959} for a short direct proof of it.
Furthermore, we have
\begin{equation*}
b_{\nu} = -\frac1{a_0^2} a_{\nu} - \frac1{a_0} \sum_{j=1}^{\nu-1}
a_j b_{\nu-j} \ge - \frac1{a_0^2} a_{\nu} \quad (\nu \ge 1).
\end{equation*}
\end{proof}

For example, consider the sum-of-divisors function $\sigma$, where
$\sigma(2^\nu)=2^{\nu+1}-1$ for every $\nu \ge 0$. The sequence
$\left(\frac1{2^{\nu+1}-1}\right)_{\nu \ge 0}$ is log-convex. This
property allows us to apply Lemma \ref{lemma_Kaluza} to obtain the
estimate of Theorem \ref{Th_sum_1_sigma} for the alternating sum
$\sum_{n\le x} (-1)^{n-1} \frac1{\sigma(n)}$.

\subsection{Kendall's renewal theorem}

Another related result is Kendall's renewal theorem. Disregarding the probabilistic context, it can be stated as
follows. See Berenhaut, Allen, and Fraser \cite[Thm.\ 1.1]{Ber2006}.

\begin{lemma} \label{lemma_Kendall} Let $\sum_{\nu =0}^{\infty} a_\nu x^\nu$
be a power series  such that $(a_\nu)_{\nu \ge 0}$ is nonincreasing, $a_0=1$, $a_\nu \ge 0$ ($\nu \ge 1$) and
$a_\nu \ll q^\nu$ as $\nu \to \infty$, where $0<q<1$ is a real number. Then
there exists $0<s<1$, $s$ real, such that for the coefficients $b_\nu$ of the reciprocal power series
one has $b_\nu \ll s^\nu$ as $\nu \to \infty$.
\end{lemma}

We deduce the next result:

\begin{corollary} \label{Cor_1_per_f} Let $f$ be a positive multiplicative function.
Assume that

(i) asymptotic formula \eqref{asympt_D_f_E_f} is valid with $1\ll R_{1/f}(x)\ll x^{\varepsilon}$ as $x\to \infty$, for every $\varepsilon >0$;

(ii) the sequence $(1/f(2^\nu))_{\nu \ge 0}$ is nonincreasing and  $1/f(2^\nu)\ll q^\nu$ as $\nu \to \infty$, where $0<q<1$ is a
real number.

Then the asymptotic formula \eqref{asympt_altern_1_per_f} holds for $\sum_{n\le x} (-1)^{n-1} \frac1{f(n)}$, with error term
\begin{equation*}
T_{1/f}(x) = x^{-u} \max(\log x, R_{1/f}(x))\ll x^{-u_1}
\end{equation*}
for some $u,u_1>0$.
\end{corollary}

\begin{proof} This is a direct consequence of Proposition \ref{Prop_1_per_f} and Lemma \ref{lemma_Kendall}, applied for
$a_{\nu}=\frac1{f(2^{\nu})}$. Note that the radius of convergence of the series $S_{1/f}(x)$ is $>1$ by condition (ii).
\end{proof}

In the case of the sum-of-unitary-divisors function $\sigma^*$ we
have $a_\nu= \sigma^*(2^\nu)=2^\nu+1$ for every $\nu \ge 1$ and
$a_0=\sigma^*(1)=1$. The sequence $\left(\frac1{a_{\nu}}\right)_{\nu
\ge 0}$ is not log-convex. Lemma \ref{lemma_Kaluza} cannot be used
to estimate the alternating sum $\sum_{n\le x} (-1)^{n-1}
\frac1{\sigma^*(n)}$. At the same time, Corollary \ref{Cor_1_per_f}, with $t=2$ furnishes an asymptotic
formula. See Section \ref{Sect_sigma_star}.

An explicit form of Lemma \ref{lemma_Kendall} (Kendall's theorem) was proved in \cite[Thm.\ 1.2]{Ber2006}. However, it is
restricted to the values $0<q<0.32$, and cannot be applied for the above special case, where $q=1/2$. To find the optimal value of $s$ for
pairs $(A,q)$ such that $a_{\nu}\le A q^{\nu}$ ($\nu \ge 1$), not satisfying assumptions of \cite[Thm.\ 1.2]{Ber2006} was formulated
by Berenhaut, Abernathy, Fan, and Foley \cite[Open question 5.4]{Ber2008}.

We prove a new explicit Kendall-type inequality, based on the following lemma.

\begin{lemma} Let $\sum_{\nu =0}^{\infty} a_\nu x^\nu$
be a power series such that $a_0=1$.  Then
for the coefficients $b_\nu$ of the reciprocal power series
$\sum_{\nu =0}^{\infty} b_\nu x^\nu$ one has $b_0=1$ and
\begin{equation} \label{b_nu}
b_{\nu} = \sum_{k=1}^{\nu} (-1)^k \sum_{\substack{j_1,\ldots,j_k\ge 1\\ j_1+\cdots +j_k=\nu}} a_{j_1}\cdots a_{j_k}
\end{equation}
\begin{equation} \label{multinom}
= \sum_{\substack{t_1,\ldots,t_{\nu} \ge 0\\ t_1+2t_2+\cdots + \nu t_{\nu}=\nu }} (-1)^{t_1+\cdots + t_{\nu}}
\binom{t_1+\cdots +t_{\nu}}{t_1,\ldots,t_{\nu}} a_1^{t_1} \cdots a_{\nu}^{t_{\nu}}
\end{equation}
for every $\nu \ge 1$, where $\binom{t_1+\cdots +t_{\nu}}{t_1,\ldots,t_{\nu}}$ are the multinomial coefficients.
\end{lemma}

Here formula \eqref{multinom} is well known, and it has been recovered several times. See, e.g., \cite[Lemma 4]{Mer2013}. However, we were not able to
find its equivalent version \eqref{b_nu} in the literature. For the sake of completeness we present their proofs.

\begin{proof} Using the geometric series formula $(1+x)^{-1}=\sum_{n=0}^{\infty} (-1)^n x^n$ and the multinomial theorem, we immediately have
\begin{equation*}
\left(1+ \sum_{\nu=1}^{\infty} a_{\nu} x^{\nu} \right)^{-1} = \sum_{t=0}^{\infty} (-1)^t \left(\sum_{\nu=1}^{\infty} a_{\nu} x^{\nu} \right)^t
\end{equation*}
\begin{equation*}
=  \sum_{\nu=0}^{\infty} x^{\nu} \sum_{\substack{t_1,\ldots,t_{\nu} \ge 0\\ t_1+2t_2+\cdots +\nu t_{\nu}=\nu}}  (-1)^{t_1+\cdots +t_{\nu}}
\binom{t_1+\cdots +t_{\nu}}{t_1,\ldots,t_{\nu}} a_1^{t_1} \cdots a_{\nu}^{t_{\nu}},
\end{equation*}
giving \eqref{multinom}. Furthermore, fix $\nu \ge 1$. By grouping the terms in \eqref{multinom} according to the values $k=t_1+\cdots +t_{\nu}$,
where $1\le k\le \nu$,  we have
\begin{equation*}
b_{\nu}= \sum_{k=1}^{\nu} (-1)^k \sum_{\substack{t_1,\ldots,t_{\nu} \ge 0\\ t_1+2t_2+\cdots + \nu t_{\nu}=\nu \\ t_1+\cdots +t_{\nu}=k}}
\binom{t_1+\cdots +t_{\nu}}{t_1,\ldots,t_{\nu}} a_1^{t_1} \cdots a_{\nu}^{t_{\nu}}.
\end{equation*}

Now, identity \eqref{b_nu} follows if we show that
\begin{equation} \label{comb_id}
\sum_{\substack{j_1,\ldots,j_k\ge 1\\ j_1+\cdots +j_k=\nu}} a_{j_1}\cdots a_{j_k}= \sum_{\substack{t_1,\ldots,t_{\nu} \ge 0\\ t_1+2t_2+\cdots + \nu t_{\nu}=\nu \\ t_1+\cdots +t_{\nu}=k}}
\binom{t_1+\cdots +t_{\nu}}{t_1,\ldots,t_{\nu}} a_1^{t_1} \cdots a_{\nu}^{t_{\nu}}.
\end{equation}

But \eqref{comb_id} is immediate by starting with its left-hand side and denoting by $t_1,\ldots, t_{\nu}$ the number of values
$j_1,\ldots,j_k$ which are equal to $1,\ldots,\nu$, respectively.
\end{proof}

\begin{proposition} \label{Prop_expl}
Assume that $\sum_{\nu =0}^{\infty} a_\nu x^\nu$ is a power series such that $a_0=1$ and
$|a_{\nu}|\le A q^{\nu}$ ($\nu \ge 1$) for some absolute constants
$A,q> 0$. Then for the coefficients $b_\nu$ of the reciprocal power series one has
\begin{equation} \label{explicit_ineq}
|b_{\nu}| \le  A q^{\nu} (A + 1)^{\nu-1} \quad (\nu \ge 1).
\end{equation}
\end{proposition}

\begin{proof} By identity \eqref{b_nu} and the assumption $|a_{\nu}|\le A q^{\nu}$ ($\nu \ge 1$) we immediately have
\begin{equation*}
|b_{\nu}| \le \sum_{k=1}^{\nu} A^k \sum_{\substack{j_1,\ldots,j_k\ge 1\\ j_1+\ldots +j_k=\nu}} q^{j_1+\cdots +j_k}
= q^{\nu} \sum_{k=1}^{\nu} A^k \sum_{\substack{j_1,\ldots,j_k\ge 1\\ j_1+\ldots +j_k=\nu}} 1
\end{equation*}
\begin{equation*}
= q^{\nu} \sum_{k=1}^{\nu} A^k \binom{\nu-1}{k-1} = A q^{\nu} (A + 1)^{\nu -1},
\end{equation*}
as asserted.
\end{proof}

Note that \eqref{explicit_ineq} is an explicit Kendall-type inequality provided that $q(A+1)<1$,
in particular if $q\le 1/2$ and $A<1$.

\begin{corollary} \label{Cor_explicit} Let $f$ be a positive multiplicative function such that

(i) asymptotic formula \eqref{asympt_D_f_E_f} is valid with $1\ll R_{1/f}(x)=o(x)$ as $x\to \infty$;

(ii) $1/f(2^\nu)\le A q^\nu$ ($\nu \ge 1$), where $A,q>0$ are fixed real constants satisfying $M:=q(A+1)<1$.

Then the asymptotic formula \eqref{asympt_altern_1_per_f} holds for $\sum_{n\le x} (-1)^{n-1} \frac1{f(n)}$, with error term \eqref{error_M}.
\end{corollary}

\begin{proof} This follows from Propositions \ref{Prop_1_per_f} and \ref{Prop_expl}. Note that the radius of convergence of the series
$S_{1/f}(x)$ is $>1$ by condition (ii).
\end{proof}

We will apply Corollary \ref{Cor_explicit} for the sum-of-bi-unitary-divisors function $\sigma^{**}$. See Section \ref{Sect_sigma**}.

\section{Results for classical functions} \label{Section_3}

In this section, we investigate alternating sums for classical multiplicative functions. We refer to Apostol \cite{Apo1976}, Hildebrand
\cite{Hil2013}, and McCarthy \cite{McC1986} for the basic properties of these functions. See Gould and Shonhiwa \cite{GouSho2008} for a list of
Dirichlet series of special arithmetic functions.

\subsection{Euler's totient function}

First consider Euler's $\varphi$ function, where $\varphi(n)=n\prod_{p\mid n} \left(1-\frac1{p} \right)$ ($n\ge 1$).

\begin{proposition}
\begin{equation} \label{D_altern_varphi}
\sum_{n=1}^{\infty} (-1)^{n-1} \frac{\varphi(n)}{n^s} =
\frac{2^s-3}{2^s-1} \cdot \frac{\zeta(s-1)}{\zeta(s)} \quad (\Re s>2).
\end{equation}
\end{proposition}

\begin{proof} This was explained in Section \ref{Sect_Alter_Dir_series}, formula
\eqref{D_altern_varphi} follows at once from
\eqref{D_altern_varphi_first}.
\end{proof}

\begin{theorem}
\begin{equation} \label{sum_varphi}
\sum_{n\le x} (-1)^{n-1} \varphi(n) = \frac1{\pi^2} x^2 + O\left( x
(\log x)^{2/3} (\log \log x)^{4/3} \right).
\end{equation}
\end{theorem}

\begin{proof} Apply Proposition \ref{Prop_f} for $f=\varphi$. It is known that
\begin{equation*}
\sum_{n\le x} \varphi(n) = \frac{3}{\pi^2} x^2 + O\left( x (\log x)^{2/3} (\log \log x)^{4/3} \right),
\end{equation*}
which is the best error term known to date, due to Walfisz \cite[Satz 1,
p.\ 144]{Wal1963}. Furthermore,
\begin{equation*}
S_{\varphi}(x)= \sum_{\nu =0}^{\infty} \varphi(2^\nu) x^\nu =
1+\sum_{\nu =1}^{\infty} 2^{\nu -1} x^{\nu} =\frac{1-x}{1-2x} \quad
\left(|x|<\frac1{2}\right)
\end{equation*}
(also see \eqref{varphi_p}). We obtain that the reciprocal power series is
\begin{equation*}
\overline{S}_{\varphi}(x) = \frac{1-2x}{1-x}= 1- \sum_{\nu
=1}^{\infty} x^\nu  \quad (|x|<1),
\end{equation*}
for which the coefficients are $b_0=1$, $b_\nu =-1$ ($\nu \ge 1$),
forming a bounded sequence. The coefficient of the main term in
\eqref{sum_varphi} is
\begin{equation*}
C_{\varphi} \left(\frac{2}{S_{\varphi}(1/4)} -1\right)= \frac{3}{\pi^2}\cdot \frac1{3}=\frac1{\pi^2}.
\end{equation*}
\end{proof}

\begin{remark} To find the corresponding constant to be multiplied by
$C_{\varphi}=3/\pi^2$ observe that by \eqref{convo}, \eqref{const}
and \eqref{D_altern_varphi},
\begin{equation*}
\frac{2}{S_{\varphi}(1/4)} -1 = \left[ \frac{2^s-3}{2^s-1} \right]_{s=2}=\frac1{3},
\end{equation*}
and similarly for other classical multiplicative functions, if we
have the representation of their alternating Dirichlet series.
\end{remark}

\begin{theorem}
\begin{equation} \label{sum_1_varphi}
\sum_{n\le x} (-1)^{n-1} \frac1{\varphi(n)} = - \frac{A}{3}
\left(\log x+\gamma -B -\frac{8}{3}\log 2\right)  + O\left( x^{-1}
(\log x)^{5/3} \right),
\end{equation}
where $\gamma$ is Euler's constant and
\begin{equation} \label{def_AB}
A=\frac{\zeta(2)\zeta(3)}{\zeta(6)}=\frac{315\zeta(3)}{2\pi^4}, \qquad B=\sum_{p\in \P} \frac{\log p}{p^2-p+1}.
\end{equation}
\end{theorem}

The result \eqref{sum_1_varphi} improves the error term $O\left(
x^{-1} (\log x)^3 \right)$ obtained by Bordell\`{e}s and Cloitre
\cite[Cor.\ 4, (i)]{BorClo2013}.

\begin{proof} Apply Proposition \ref{Prop_1_per_f} for $f=\varphi$. The asymptotic formula
\begin{equation*}
\sum_{n\le x} \frac1{\varphi(n)} = A \left(\log
x+\gamma - B \right)  + O\left( x^{-1} (\log
x)^{2/3} \right)
\end{equation*}
with constants $A$ and $B$ defined by \eqref{def_AB} and with
the weaker error term $O\left( x^{-1} \log x \right)$ goes back to
the work of Landau. See \cite[Thm.\ 1.1]{DeKIvi1980}. The error
term above was obtained by Sitaramachandrarao \cite{Sit1982}.

Now
\begin{equation*}
S_{1/\varphi}(x)= \sum_{\nu=0}^{\infty} \frac{x^\nu}{\varphi(2^\nu)}
= 1+ \sum_{\nu=1}^{\infty} \frac{x^\nu}{2^{\nu-1}} =\frac{2+x}{2-x}
\quad (|x|<2),
\end{equation*}
\begin{equation*}
\overline{S}_{1/\varphi}(x) = \frac{1-x/2}{1+x/2}= 1+ \sum_{\nu
=1}^{\infty} (-1)^\nu \frac{x^\nu}{2^{\nu-1}} \quad (|x|<2);
\end{equation*}
hence $b_{\nu} \ll 2^{-\nu}$ and choose $M=1/2$. Using that $S_{1/\varphi}(1)=3$ and $S'_{1/\varphi}(1)=4$,
the proof is complete.
\end{proof}

\subsection{Dedekind function}

The Dedekind function $\psi$ is given by $\psi(n)=n \prod_{p\mid n}
\left(1+\frac1{p}\right)$ ($n\ge 1$).

\begin{proposition}
\begin{equation} \label{D_altern_psi}
\sum_{n=1}^{\infty} (-1)^{n-1} \frac{\psi(n)}{n^s} =
\frac{2^s-5}{2^s+1} \cdot \frac{\zeta(s)\zeta(s-1)}{\zeta(2s)} \quad
(\Re s>2),
\end{equation}
\end{proposition}

\begin{proof} It is well-known that
\begin{equation*}
\sum_{n=1}^{\infty} \frac{\psi(n)}{n^s} =
\frac{\zeta(s)\zeta(s-1)}{\zeta(2s)} \quad (\Re s>2),
\end{equation*}
and \eqref{D_altern_psi} follows like \eqref{D_altern_varphi}, by using Proposition
\ref{prop_1}.
\end{proof}

\begin{theorem}
\begin{equation} \label{sum_psi}
\sum_{n\le x} (-1)^{n-1} \psi(n) = - \frac{3}{2\pi^2} x^2 + O\left(
x (\log x)^{2/3} \right),
\end{equation}
\end{theorem}

\begin{proof} Apply Proposition \ref{Prop_f} for $f=\psi$. It is known that
\begin{equation*}
\sum_{n\le x} \psi(n) = \frac{15}{2\pi^2} x^2 + O\left( x (\log
x)^{2/3} \right),
\end{equation*}
the best estimate
up to now. See Walfisz \cite[Satz 2, p.\ 100]{Wal1963}. Here
\begin{equation*}
S_{\psi}(x)= \sum_{\nu=0}^{\infty} \psi(2^\nu) x^\nu  = 1+ 3
\sum_{\nu=1}^{\infty} 2^{\nu-1} x^n =\frac{1+x}{1-2x} \quad
\left(|x|<\frac1{2}\right).
\end{equation*}

We obtain that the reciprocal power series is
\begin{equation*}
\overline{S}_{\psi}(x) = \frac{1-2x}{1+x}= 1 + 3
\sum_{\nu=1}^{\infty} (-1)^\nu x^\nu \quad (|x|<1),
\end{equation*}
for which the coefficients are $b_0=1$, $b_\nu=3(-1)^\nu$ ($\nu \ge
1$), forming a bounded sequence. The coefficient of the main term in
\eqref{sum_psi} is
\begin{equation*}
C_{\psi} \left(\frac{2}{S_{\psi}(1/4)} -1\right)= \frac{15}{2\pi^2}
\cdot (-\frac1{5})= - \frac{3}{2\pi^2}.
\end{equation*}
\end{proof}

\begin{theorem}
\begin{equation} \label{sum_1_psi}
\sum_{n\le x} (-1)^{n-1} \frac1{\psi(n)} = \frac{C}{5} \left(\log x+
\gamma + D + \frac{24}{5}\log 2\right)  + O\left( x^{-1} (\log
x)^{2/3}(\log \log x)^{4/3} \right),
\end{equation}
where $\gamma$ is Euler's constant and
\begin{equation} \label{def_C}
C=\prod_{p\in \P} \left(1-\frac1{p(p+1)} \right), \qquad D=\sum_{p\in
\P} \frac{\log p}{p^2+p-1}.
\end{equation}
\end{theorem}

The result \eqref{sum_1_psi} improves the error term $O\left( x^{-1}
(\log x)^2 \right)$ obtained by Bordell\`{e}s and Cloitre
\cite[Cor.\ 4, (iii)]{BorClo2013}. The constant $C \doteq 0.704442$
is sometimes called the carefree constant, and its digits form
the sequence \seqnum{A065463}
in Sloane's Online Encyclopedia of Integer Sequences (OEIS)
\cite{OEIS}. Also see Finch \cite[Sect. 2.5.1]{Finc2003}.

\begin{proof} Apply Proposition \ref{Prop_1_per_f} for $f=\psi$. The asymptotic formula
\begin{equation*}
\sum_{n\le x} \frac1{\psi(n)} = C \left(\log x+\gamma + D \right)  +
O\left( x^{-1} (\log x)^{2/3} (\log \log x)^{4/3}\right),
\end{equation*}
where $C$ and $D$ are the constants given by \eqref{def_C}, is due
to Sita Ramaiah and Suryanarayana \cite[Cor.\ 4.2]{SitSur1979}.

Furthermore,
\begin{equation*}
S_{1/\psi}(x)= \sum_{\nu=0}^{\infty} \frac{x^\nu}{\psi(2^\nu)} = 1+
\frac{2}{3} \sum_{\nu=1}^{\infty} \frac{x^\nu}{2^{\nu}}
=\frac{6-x}{3(2-x)} \quad (|x|<2),
\end{equation*}
\begin{equation*}
\overline{S}_{1/\psi}(x) = \frac{1-x/2}{1-x/6}= 1-2\sum_{\nu
=1}^{\infty} \frac{x^\nu}{6^{\nu}} \quad (|x|<6),
\end{equation*}
which shows that
\begin{equation*}
b_{\nu} = - \frac{2}{6^{\nu}} \quad (\nu \ge 1).
\end{equation*}

Hence $b_{\nu} \ll 6^{-\nu}$ and choose $M=1/6$. Using that $S_{1/\psi}(1)=\frac{5}{3}$ and
$S'_{1/\psi}(1)=\frac{4}{3}$, we deduce \eqref{sum_1_psi}.
\end{proof}

\subsection{Sum-of-divisors function}

Consider the function $\sigma(n)=\sum_{d\mid n} d$ ($n\ge 1$).

\begin{proposition}
\begin{equation} \label{D_altern_sigma}
\sum_{n=1}^{\infty} (-1)^{n-1} \frac{\sigma(n)}{n^s}  =
\left(1-\frac{6}{2^s}+\frac{4}{2^{2s}}\right) \zeta(s)\zeta(s-1)
\quad (\Re s>2).
\end{equation}
\end{proposition}

Note that $(-1)^{n-1}\sigma(n)$ is sequence \seqnum{A143348} in the OEIS \cite{OEIS}, where identity \eqref{D_altern_sigma} is given.

\begin{proof} Use the familiar formula
\begin{equation*}
\sum_{n=1}^{\infty} \frac{\sigma(n)}{n^s}  = \zeta(s)\zeta(s-1)
\quad (\Re s>2)
\end{equation*}
and Proposition \ref{prop_1}.
\end{proof}

\begin{theorem}
\begin{equation} \label{sum_sigma}
\sum_{n\le x} (-1)^{n-1} \sigma(n) = - \frac{\pi^2}{48} x^2 +
O\left( x (\log x)^{2/3} \right).
\end{equation}
\end{theorem}

\begin{proof} Apply Proposition \ref{Prop_f} for $f=\sigma$. It is known that
\begin{equation*}
\sum_{n\le x} \sigma(n) = \frac{\pi^2}{12} x^2 + O\left( x (\log
x)^{2/3} \right),
\end{equation*}
the best up to now. See Walfisz \cite[Satz 4, p.\
99]{Wal1963}. Here
\begin{equation*}
S_{\sigma}(x)= \sum_{\nu=0}^{\infty} \sigma(2^\nu) x^\nu  =
\sum_{\nu=0}^{\infty} (2^{\nu+1}-1) x^n =\frac1{(1-x)(1-2x)} \quad
\left(|x|< \frac1{2} \right).
\end{equation*}

Hence
\begin{equation*}
\overline{S}_{\sigma}(x)= (1-x)(1-2x)= 1-3x+2x^2,
\end{equation*}
for which the 
coefficients are $b_0=1$, $b_1 =-3$, $b_2=2$, $b_{\nu}=0$ ($\nu \ge 3$).
The coefficient of the main term in \eqref{sum_sigma} is from \eqref{D_altern_sigma},
\begin{equation*}
\frac{\pi^2}{12} \left[1-\frac{6}{2^s}+
\frac{4}{2^{2s}}\right]_{s=2} =- \frac{\pi^2}{48}.
\end{equation*}
\end{proof}

The following asymptotic formula is due to Sita~Ramaiah and
Suryanarayana \cite[Cor.\ 4.1]{SitSur1979}:
\begin{equation} \label{asymp_1_per_sigma}
\sum_{n\le x} \frac1{\sigma(n)} = E \left(\log x + \gamma + F
\right) + O\left( x^{-1} (\log x)^{2/3}(\log \log x)^{4/3} \right),
\end{equation}
where $\gamma$ is Euler's constant,
\begin{equation*}
E =\prod_{p\in \P} \alpha(p), \qquad  F= \sum_{p\in \P}
\frac{(p-1)^2 \beta(p)\log p}{p\alpha(p)},
\end{equation*}
\begin{equation*}
\alpha(p) = \left(1-\frac1{p} \right) \sum_{\nu=0}^{\infty} \frac1{\sigma(p^\nu)}=  1- \frac{(p-1)^2}{p} \sum_{j=1}^{\infty}
\frac1{(p^j-1)(p^{j+1}-1)},
\end{equation*}
\begin{equation*}
\beta(p) = \sum_{j=1}^{\infty} \frac{j}{(p^j-1)(p^{j+1}-1)}.
\end{equation*}

We prove the next result:
\begin{theorem} \label{Th_sum_1_sigma}
\begin{equation} \label{sum_1_sigma}
\sum_{n\le x} (-1)^{n-1} \frac1{\sigma(n)} = E\left( \left(\frac2{K} -1
\right) \left(\log x+ \gamma + F \right) +2(\log 2) \frac{K'}{K^2}\right)
\end{equation}
\begin{equation*}
+ O\left( x^{-1} (\log x)^{5/3}(\log \log x)^{4/3} \right),
\end{equation*}
where
\begin{equation} \label{const_Erdos}
K= \sum_{j=0}^{\infty} \frac1{2^{j+1}-1}, \qquad K'=
\sum_{j=1}^{\infty} \frac{j}{2^{j+1}-1}.
\end{equation}
\end{theorem}

The result \eqref{sum_1_sigma} improves the error term $O\left(
x^{-1} (\log x)^4 \right)$ obtained by Bordell\`{e}s and Cloitre
\cite[Cor.\ 4, (v)]{BorClo2013}. Here $K \doteq 1.606695$ is the
Erd\H{o}s-Borwein constant, known to be irrational. See sequence
\seqnum{A065442} in the OEIS \cite{OEIS}.

\begin{proof} Apply Proposition \ref{Prop_1_per_f} for $f=\sigma$, using formula
\eqref{asymp_1_per_sigma}. Now
\begin{equation*}
S_{1/\sigma}(x)= \sum_{\nu=0}^{\infty} \frac{x^\nu}{\sigma(2^\nu)} =
\sum_{\nu=0}^{\infty} \frac{x^\nu}{2^{\nu+1}-1}
\end{equation*}
and $S_{1/\sigma}(1)= K$. Note that
$S'_{1/\psi}(1)=K'$, given above.

The coefficients $b_{\nu}$ of the reciprocal power series are
$b_0=1$, $b_1=-\frac1{3}$, $b_2=-\frac{2}{63}$, $b_3=-\frac{8}{945}$, etc. Observe that
the sequence $\left(\frac1{2^{\nu+1}- 1}\right)_{\nu \ge 0}$
is log-convex.  Therefore, according to Lemma \ref{lemma_Kaluza},
\begin{equation*}
- \frac1{2^{\nu+1}-1} \le b_{\nu} \le 0 \quad (\nu\ge 1),
\end{equation*}
which shows that $b_{\nu}\ll 2^{-\nu}$ and we can choose $M=1/2$.
\end{proof}

\subsection{Divisor function}

Now consider another classical function, the divisor function
$\tau(n)=\sum_{d\mid n} 1$ ($n\ge 1$). Using the familiar formula
\begin{equation*}
\sum_{n=1}^{\infty} \frac{\tau(n)}{n^s}
= \zeta^2(s) \quad (\Re s>1),
\end{equation*}
and Proposition \ref{prop_1} we deduce
\begin{proposition}
\begin{equation*}
\sum_{n=1}^{\infty} (-1)^{n-1} \frac{\tau(n)}{n^s} =
\left(1-\frac{4}{2^s}+\frac{2}{2^{2s}}\right) \zeta^2(s) \quad (\Re
s>1).
\end{equation*}
\end{proposition}

By similar considerations we also have

\begin{proposition}
\begin{equation*}
\sum_{n=1}^{\infty} (-1)^{n-1} \frac1{\tau(n)n^s}  =
\left(\frac1{2^{s-1}}\left(\log
\left(1-\frac1{2^s}\right)\right)^{-1} +1 \right) \prod_{p\in \P}
p^s \log \left(1-\frac1{p^s}\right) \quad (\Re s>1).
\end{equation*}
\end{proposition}

\begin{theorem}
\begin{equation*}
\sum_{n\le x} (-1)^{n-1} \tau(n) = - \frac1{2} x\log x + \left(\frac1{2}-\gamma + \log 2\right) x + O\left( x^{\theta +\varepsilon}
\right),
\end{equation*}
where $\theta$ is the best exponent in Dirichlet's divisor problem.
\end{theorem}

\begin{proof} Proposition \ref{Prop_f} cannot be applied. Using that
\begin{equation*}
\sum_{n\le x} \tau(n) = x\log x + (2\gamma -1)x + O\left( x^{\theta
+\varepsilon} \right)
\end{equation*}
the result follows by similar arguments.
\end{proof}

Note that the actual best result for $\theta$ is $\theta = 131/416 \doteq 0.314903$, due to
Huxley \cite{Hux2003}.

Now we consider the following result, which goes back to the work of
Ramanujan \cite[Eq.\ (7)]{Ram1916}. See Wilson \cite[Sect.\
3]{Wil1922} for its proof:
\begin{equation*}
\sum_{n\le x} \frac1{\tau(n)} = x \sum_{j=1}^N \frac{A_j}{(\log x)^{j-1/2}}  + O\left( \frac{x}{(\log x)^{N+1/2}} \right),
\end{equation*}
valid for every real $x\ge 2$ and every fixed integer $N\ge 1$ where $A_j$ ($1\le j \le N$) are computable constants,
\begin{equation*}
A_1 = \frac1{\sqrt{\pi}} \prod_{p\in \P} \left(\sqrt{p^2-p} \ \log \left(\frac{p}{p-1}\right) \right).
\end{equation*}

We prove

\begin{theorem} \label{th_1_per_tau}
\begin{equation*}
\sum_{n\le x} (-1)^{n-1} \frac1{\tau(n)} = x \sum_{t=1}^N \frac{B_t}{(\log x)^{t-1/2}}  + O\left( \frac{x}{(\log x)^{N+1/2}} \right),
\end{equation*}
valid for every real $x\ge 2$ and every fixed integer $N\ge 1$ where $B_t$ ($1\le t \le N$) are computable constants. In particular,
\begin{equation*}
B_1 = A_1 \left(\frac1{\log 2}-1 \right).
\end{equation*}
\end{theorem}

\begin{proof}
Now
\begin{equation*}
S_{1/\tau}(x)=\sum_{\nu=0}^{\infty} \frac1{\tau(2^{\nu})} x^{\nu}= \sum_{\nu=0}^{\infty} \frac1{\nu+1} x^{\nu}= -\frac{\log (1-x)}{x} \ \quad (|x|<1)
\end{equation*}
and the reciprocal power series is
\begin{equation*}
\overline{S}_{1/\tau}(x) = - \frac{x}{\log (1-x)} = \sum_{\nu
=0}^{\infty} b_{\nu}x^{\nu},
\end{equation*}
where $b_0=1$, $b_1=-1/2$, $b_2=-1/12$, $b_3=-1/24$, etc.
Note that the sequence $\left(\frac1{\nu+1}\right)_{\nu \ge 0}$ is log-convex.
According to Lemma \ref{lemma_Kaluza} (this example was considered by
Kaluza \cite{Kal1928}),
\begin{equation*}
- \frac1{\nu+1} \le b_{\nu} \le 0 \qquad (\nu \ge 1).
\end{equation*}

This shows, using \eqref{sum}, that
\begin{equation*}
(-1)^{n-1} \frac1{\tau(n)} = \sum_{dj=n} h_{1/\tau}(d) \frac1{\tau(j)} \qquad (n \ge 1),
\end{equation*}
where the function $h_{1/\tau}$ is multiplicative, $h_{1/\tau}(2^{\nu})\ll \frac1{\nu}$ as $\nu \to \infty$ and $h_{1/\tau}(p^\nu)=0$ for
every prime $p>2$ and $\nu \ge 1$.

Hence
\begin{equation*}
T(x):= \sum_{n\le x} (-1)^{n-1} \frac1{\tau(n)} = \sum_{d\le x/2} h_{1/\tau}(d) \sum_{j\le x/d} \frac1{\tau(j)} + \sum_{x/2<d\le x} h_{1/\tau}(d)
\end{equation*}
\begin{equation*}
= \sum_{d\le x/2} h_{1/\tau}(d) \left(\frac{x}{d} \sum_{j=1}^N \frac{A_j}{(\log (x/d))^{j-1/2}}  + O\left( \frac{x/d}{(\log (x/d))^{N+1/2}}
\right)\right)
+ \sum_{x/2<d\le x} h_{1/\tau}(d)
\end{equation*}
\begin{gather*}
= x \sum_{j=1}^N \frac{A_j}{(\log x)^{j-1/2}}  \sum_{d\le x/2} \frac{h_{1/\tau}(d)}{d (1-\frac{\log d}{\log x})^{j-1/2}}
+ O\left( \frac{x}{(\log x)^{N+1/2}} \sum_{d\le x/2} \frac{|h_{1/\tau}(d)|}{d (1-\frac{\log d}{\log x})^{N+1/2}} \right)
\\ + \sum_{x/2<d\le x} h_{1/\tau}(d).
\end{gather*}

Here the last term is small:
\begin{equation*}
\sum_{x/2< d\le x} h_{1/\tau}(d) \ll \sum_{d=2^{\nu}\le x}  |h_{1/\tau}(2^{\nu})| \ll
\sum_{\nu \le \log x/\log 2}  \frac1{\nu} \ll \log \log x.
\end{equation*}

Using the power series expansion
\begin{equation*}
(1+x)^t = \sum_{j=0}^{\infty} \binom{t}{j} x^j  \qquad (x,t\in \R, |x|<1),
\end{equation*}
we deduce
\begin{gather*}
\sum_{d\le x/2} \frac{|h_{1/\tau}(d)|}{d (1-\frac{\log d}{\log x})^{N+1/2}} = \sum_{d\le x/2} \frac{|h_{1/\tau}(d)|}{d}
\left(1+O\left(\frac{\log d}{\log x} \right) \right) \\
= \sum_{d=2^{\nu}\le x/2} \frac{|h_{1/\tau}(2^{\nu})|}{2^{\nu}} + O\left(\frac1{\log x} \sum_{d=2^{\nu} \le x/2}
\frac{|h_{1/\tau}(2^{\nu})|}{2^{\nu}} \log 2^{\nu} \right) \\
\ll \sum_{2^\nu \le x/2} \frac1{\nu 2^{\nu}} + \frac1{\log x} \sum_{2^\nu \le x/2}
\frac1{2^{\nu}} \ll 1.
\end{gather*}

Therefore, the remainder term of above is
\begin{gather*}
O \left( \frac{x}{(\log x)^{N+1/2}} \right).
\end{gather*}

Furthermore,
\begin{gather*}
\sum_{d\le x/2} \frac{h_{1/\tau}(d)}{d (1-\frac{\log d}{\log x})^{j-1/2}}
= \sum_{d\le x/2} \frac{h_{1/\tau}(d)}{d} \sum_{\ell=0}^{\infty} (-1)^{\ell} \binom{-j+1/2}{\ell} \left(\frac{\log d}{\log x}\right)^{\ell}
\\ = \sum_{\ell=0}^{\infty} (-1)^{\ell} \binom{-j+1/2}{\ell} \frac1{(\log x)^{\ell}} \sum_{d\le x/2} \frac{h_{1/\tau}(d)}{d} (\log d)^{\ell}
\\ = \sum_{\ell=0}^{\infty} (-1)^{\ell} \binom{-j+1/2}{\ell} \frac1{(\log x)^{\ell}} \left(K_{\ell} +
O \left(\frac{(\log x)^{\ell-1}}{x} \right) \right),
\end{gather*}
where for every $\ell \ge 0$ the series
\begin{gather*}
K_{\ell}: = \sum_{d=1}^{\infty} \frac{h_{1/\tau}(d)}{d} (\log d)^{\ell} = \sum_{\substack{d=2^{\nu} \\ \nu \ge 0}}
\frac{h_{1/\tau}(2^{\nu})}{2^\nu} (\log 2^\nu)^{\ell}
\end{gather*}
is absolutely convergent, since  $|h_{1/\tau}(2^{\nu})|\ll \frac1{\nu}$, and
\begin{gather*}
\sum_{d>x/2} \frac{|h_{1/\tau}(d)|}{d} (\log d)^{\ell} =
\sum_{d=2^{\nu} > x/2} \frac{|h_{1/\tau}(2^{\nu})|}{2^\nu} (\log 2^\nu)^{\ell} \ll  \sum_{\nu > \log x/\log 2} \frac{\nu^{\ell-1}}{2^{\nu}} \ll
\frac{(\log x)^{\ell-1}}{x}.
\end{gather*}

We deduce that
\begin{gather*}
T(x) = x \sum_{j=1}^N \frac{A_j}{(\log x)^{j-1/2}} \sum_{\ell=0}^{\infty} (-1)^{\ell} \binom{-j+1/2}{\ell} \frac1{(\log x)^{\ell}}
\left(K_{\ell} + O \left(\frac{(\log x)^{\ell-1}}{x} \right) \right)
\\ + O\left( \frac{x}{(\log x)^{N+1/2}} \right) \\
= x \sum_{t=1}^N \frac1{(\log x)^{t-1/2}} \sum_{j=1}^N (-1)^{t-j} \binom{-j+1/2}{t-j}  A_jK_{t-j}
+  O\left( \frac{x}{(\log x)^{N+1/2}} \right).
\end{gather*}

The proof is complete by denoting
\begin{equation*}
B_t = \sum_{j=1}^N (-1)^{t-j} \binom{-j+1/2}{t-j} A_j K_{t-j},
\end{equation*}
where $B_1=A_1K_0=A_1(\frac1{\log 2}-1)$ by \eqref{const} (applied for $s=1$).
\end{proof}

Note that a similar asymptotic formula can be deduced for the alternating sum
\begin{equation*}
\sum_{n\le x} (-1)^{n-1} \frac1{\tau_k(n)},
\end{equation*}
where $\tau_k(n)$ is the Piltz divisor function, based on the result for $\sum_{n\le x}
\frac1{\tau_k(n)}$, due to De~Koninck and Ivi\'c \cite[Thm.\
1.2]{DeKIvi1980}.

\subsection{Gcd-sum function} \label{Subsection_gcd}

Let $P(n)=\sum_{k=1}^n \gcd(k,n)$ be the gcd-sum function. 
Known results include the following:
$P$ is multiplicative, $P(p^\nu)=p^{\nu-1}(p(\nu+1)-\nu)$ ($\nu \ge 1$),
\begin{equation*}
\sum_{n=1}^{\infty} \frac{P(n)}{n^s}  = \frac{\zeta^2(s-1)}{\zeta(s)} \quad (\Re s>2),
\end{equation*}
\begin{equation*}
\sum_{n\le x} P(n) = \frac{3}{\pi^2} x^2 \left(\log x +
2\gamma -\frac1{2}-\frac{\zeta'(2)}{\zeta(2)} \right) + O\left( x^{1+\theta+\varepsilon} \right),
\end{equation*}
where $\theta$ is the best exponent in Dirichlet's divisor problem. See the survey of the author \cite{Tot2010GCD}.
We have

\begin{proposition}
\begin{equation*}
\sum_{n=1}^{\infty} (-1)^{n-1} \frac{P(n)}{n^s} =\left(2\left(1-\frac1{2^{s-1}}\right)^2 \left(1-\frac1{2^s}\right)^{-1}  - 1\right)
\frac{\zeta^2(s-1)}{\zeta(s)} \quad (\Re s>2).
\end{equation*}
\end{proposition}

\begin{theorem}
\begin{equation*}
\sum_{n\le x} (-1)^{n-1} P(n) = - \frac1{\pi^2} x^2\left(\log x +
2\gamma -\frac1{2}-\frac{\zeta'(2)}{\zeta(2)}-\frac{10}{3}\log 2
\right) + O\left( x^{1+\theta+\varepsilon} \right),
\end{equation*}
where $\theta$ is the best exponent in Dirichlet's divisor problem.
\end{theorem}

\begin{proof} Similar to the proofs of above. Here $h_{P}(2)=-6$, $h_P(2^{\nu})=2$ ($\nu \ge 2$).
\end{proof}

The next formula was proved by Chen and Zhai \cite[Thm.\
4]{CheZha2011}, sharpening a result of the author \cite[Thm.\
6]{Tot2010GCD}:
\begin{equation*}
\sum_{n\le x} \frac1{P(n)} =  \sum_{j=0}^N \frac{K_j}{(\log x)^{j-1/2}}  + O\left( \frac1{(\log x)^{N+1/2}} \right),
\end{equation*}
valid for every real $x\ge 2$ and every fixed integer $N\ge 1$ where $K_j$ ($1\le j \le N$) are computable constants,
\begin{equation*}
K_0 = \frac{2}{\sqrt{\pi}} \prod_{p\in \P} \left(1-\frac1{p}\right)^{1/2} \sum_{\nu=0}^{\infty} \frac1{P(p^\nu)}.
\end{equation*}

We have

\begin{theorem}
\begin{equation*}
\sum_{n\le x} (-1)^{n-1} \frac1{P(n)} = \sum_{t=0}^N \frac{D_t}{(\log x)^{t-1/2}}
+ O\left( \frac1{(\log x)^{N+1/2}} \right),
\end{equation*}
valid for every real $x\ge 2$ and every fixed integer $N\ge 1$ where $D_t$ ($0\le t \le N$) are computable constants. In particular,
\begin{equation*}
D_0 = K_0 \left(\frac1{2(\log 2-1)}-1 \right).
\end{equation*}
\end{theorem}

\begin{proof} Similar to the proof of Theorem \ref{th_1_per_tau}. Here $P(2^\nu)= (\nu+2) 2^{\nu -1}$ ($\nu \ge 1$).
The crucial fact is that the sequence $\left(\frac1{(\nu+2)2^{\nu
-1}}\right)_{\nu \ge 0}$ is log-convex, and therefore Lemma
\ref{lemma_Kaluza} can be used to deduce that $h_{1/P}(2^\nu)\ll
\frac1{\nu 2^\nu}$.
\end{proof}

\subsection{Squarefree kernel} \label{Sect_Squarefree kernel}

Now we move to the function $\kappa(n)= \prod_{p\mid n} p$ ($n\ge 1$), the squarefree
kernel of $n$ (radical of $n$). It is known that
\begin{equation*}
\sum_{n=1}^{\infty} \frac{\kappa(n)}{n^s}  = \zeta(s)\prod_{p\in \P}
\left(1+\frac{p-1}{p^s} \right) \quad (\Re s>2)
\end{equation*}
and for $x\ge 3$,
\begin{equation*}
\sum_{n\le x} \kappa(n) = \frac{C}{2} x^2 + O\left( R_{\kappa}(x) \right),
\end{equation*}
where $C$ is the constant defined by \eqref{def_C}, $R_{\kappa}(x)= x^{3/2}\delta(x)$ unconditionally
and $R_{\kappa}(x)= x^{7/5}\omega(x)$ assuming the Riemann hypothesis (RH), with
\begin{equation} \label{delta}
\delta(x)= \exp\left(-c_1(\log x)^{3/5}(\log \log x)^{-1/5}\right),
\end{equation}
\begin{equation} \label{omega}
\omega(x)= \exp\left(c_2(\log x)(\log \log x)^{-1}\right),
\end{equation}
$c_1,c_2$ being positive constants.  These estimates are due (for a
more general function) to Suryanarayana and Subrahmanyam \cite[Cor.\
4.3.5 and 4.4.5]{SurSub1977}.

We have
\begin{proposition}
\begin{equation*}
\sum_{n=1}^{\infty} (-1)^{n-1} \frac{\kappa(n)}{n^s}  = \frac{2^s-3}{2^s+1} \zeta(s) \prod_{p\in \P}
\left(1+\frac{p-1}{p^s} \right) \quad (\Re s>2).
\end{equation*}
\end{proposition}

\begin{theorem} \label{Th_kappa}
\begin{equation*}
\sum_{n\le x} (-1)^{n-1} \kappa(n) = \frac{C}{10} x^2 + O\left( R_{\kappa}(x) \right),
\end{equation*}
where $C$ is given by \eqref{def_C} and $R_{\kappa}(x)$ is defined above.
\end{theorem}

\begin{proof} Here, to deduce the unconditional result, Proposition \ref{Prop_f} cannot be applied,
since the function  $\delta(x)$ is not increasing. However,
$x^{\varepsilon}\delta(x)$ is increasing for any $\varepsilon>0$ and
we obtain by \eqref{sum},
\begin{equation*}
\sum_{n\le x} (-1)^{n-1} \kappa(n) = \sum_{d\le x} h_{\kappa}(d)
\left( \frac{C}{2} \left(\frac{x}{d}\right)^2 +
O\left(\left(\frac{x}{d}\right)^{3/2} \delta(x/d) \right) \right)
\end{equation*}
\begin{equation*}
=\frac{C x^2}{2} \sum_{d\le x} \frac{h_{\kappa}(d)}{d^2} +
O\left(x^{\varepsilon}\delta(x) x^{3/2-\varepsilon} \sum_{d\le x}
\frac{|h_{\kappa}(d)|}{d^{3/2-\varepsilon}}\right).
\end{equation*}

Note that $h_{\kappa}(2^\nu)=4(-1)^\nu$ ($\nu \ge 1$). Hence the function
$h_{\kappa}$ is bounded and the result is obtained by
the usual arguments.

Assuming RH, Proposition \ref{Prop_f} can directly be applied, since $\omega(x)$ is increasing.
\end{proof}

It is known that
\begin{equation*}
K(x):= \sum_{n\le x} \frac1{\kappa(n)} = \exp
\left(\left(1+o(1)\right)\left(\frac{8\log x}{\log \log x}
\right)^{1/2}\right) \quad (x\to \infty),
\end{equation*}
due to de~Bruijn \cite{Bru1962}, confirming a conjecture of
Erd\H{o}s. In fact,
\begin{equation*}
K(x) \sim \frac1{2} e^{\gamma}F(\log x)(\log \log x) \quad (x\to \infty),
\end{equation*}
where $\gamma$ is Euler's constant and
\begin{equation*}
F(t):= \frac{6}{\pi^2} \sum_{m=1}^{\infty} \frac{\min(1,e^t/m)}{\prod_{p\mid m} (p+1)} \quad (t\ge 0),
\end{equation*}
which follows from a more precise asymptotic formula with error term, recently
established by Robert and Tenenbaum \cite[Thm.\ 4.3]{RobTen2013}. We point out that according to \cite[Eq.\ (2.12)]{RobTen2013},
there exists a sequence of polynomials $(Q_j)_{j\ge 1}$ with $\deg Q_j\le j$ ($j\ge 1$) such that for any $N\ge 1$,
\begin{equation*}
F(t)= \exp \left( \left(\frac{8t}{\log t} \right)^{1/2} \left( 1 + \sum_{j=1}^N \frac{Q_j(\log \log t)}{(\log t)^j} +
O\left(\left(\frac{\log \log t}{\log t} \right)^{N+1} \right) \right) \right)  \quad (t\ge 3).
\end{equation*}

Note that
\begin{equation*}
S_{1/\kappa}(x)= \sum_{\nu=0}^{\infty} \frac1{\kappa(2^\nu)}x^\nu =
\frac{2-x}{2(1-x)} \quad (|x|<1),
\end{equation*}
\begin{equation*}
\overline{S}_{1/\kappa}(x)= \frac{2(1-x)}{2-x} = 1 -
\sum_{\nu=1}^{\infty} \frac{x^\nu}{2^\nu} \quad (|x|<2),
\end{equation*}
therefore $h_{1/\kappa}(2^\nu)=-\frac1{2^{\nu-1}}$ ($\nu \ge 1$) and
$\sum_{n=1}^{\infty} h_{1/\kappa}(n)=-1$. It follows that
\begin{equation} \label{KK}
K_{\altern}(x):= \sum_{n\le x} (-1)^{n-1}
\frac1{\kappa(n)} =  K(x) - 2 \sum_{2\le 2^{\nu}\le x} \frac1{2^\nu}
K\left(\frac{x}{2^\nu}\right).
\end{equation}

Identity \eqref{KK} and the deep analytic results of Robert and
Tenenbaum \cite{RobTen2013} lead to the following:

\begin{theorem} {\rm (Tenenbaum \cite{Ten2016})} One has
\begin{equation} \label{KK_asymp}
K_{\altern}(x) \sim - K(x) \quad (x\to \infty)
\end{equation}
and a genuine asymptotic formula with effective remainder term may be derived for $K_{\altern}(x)$.
\end{theorem}

\subsection{Squarefree numbers}

Now consider the squarefree numbers for which the characteristic function
is $\mu^2$, where $\mu$ is the M\"obius function. It is well-known that
\begin{equation*}
\sum_{n=1}^{\infty} \frac{\mu^2(n)}{n^s}  =
\frac{\zeta(s)}{\zeta(2s)} \quad (\Re s>1)
\end{equation*}
and
\begin{equation*}
\sum_{n\le x} \mu^2(n) = \frac{6}{\pi^2} x + O\left( R_{\mu^2}(x) \right),
\end{equation*}
where $R_{\mu^2}(x)= x^{1/2}\delta(x)$, with $\delta(x)$ defined by
\eqref{delta}, unconditionally, due to Walfisz \cite[Satz 1, p.\
192]{Wal1963}, and $R_{\mu^2}(x)= x^{11/35+\varepsilon}$
($\varepsilon>0$) assuming RH, due very recently to Liu
\cite{Liu2016}.

\begin{proposition}
\begin{equation*}
\sum_{n=1}^{\infty} (-1)^{n-1} \frac{\mu^2(n)}{n^s}  =
\frac{2^s-1}{2^s+1}\cdot \frac{\zeta(s)}{\zeta(2s)} \quad (\Re s>1).
\end{equation*}
\end{proposition}

\begin{theorem}
\begin{equation*}
\sum_{n\le x} (-1)^{n-1} \mu^2(n) = \frac{2}{\pi^2} x + O\left( R_{\mu^2}(x) \right).
\end{equation*}
\end{theorem}

\begin{proof} Similar to the proof of Theorem \ref{Th_kappa}. Note that here
$h_{\mu^2}(2^\nu)=2(-1)^\nu$ ($\nu \ge 1$). Hence the function
$h_{\kappa}$ is bounded.
\end{proof}

\subsection{Number of abelian groups of a given order}

Let $a(n)$ denote, as usual, the number of abelian groups of order
$n$. This is another classical multiplicative function, investigated
by several authors. It is known that
\begin{equation*}
\sum_{n=1}^{\infty} \frac{a(n)}{n^s} = \prod_{k=1}^{\infty}
\zeta(ks) \quad (\Re s>1),
\end{equation*}
\begin{equation} \label{asympt_a_n}
\sum_{n\le x} a(n) = C_1x+ C_2x^{1/2}+C_3x^{1/3}+ O\left(
x^{1/4+\varepsilon} \right),
\end{equation}
where
\begin{equation*}
C_j=\prod_{\substack{k=1\\ k\ne j}}^{\infty} \zeta(k/j) \quad
(j=1,2,3),
\end{equation*}
this best error term to date due to Robert and Sargos
\cite{RS2006}.

We have
\begin{proposition}
\begin{equation*}
\sum_{n=1}^{\infty} (-1)^{n-1} \frac{a(n)}{n^s} =
\left(2\prod_{k=1}^{\infty}\left(1-\frac1{2^{ks}} \right) -1 \right)
\prod_{k=1}^{\infty} \zeta(ks) \quad (\Re s>1).
\end{equation*}
\end{proposition}

\begin{theorem}
\begin{equation*}
\sum_{n\le x} (-1)^{n-1} a(n) = C_1K_1 x+
C_2K_2x^{1/2}+C_3K_3x^{1/3}+ O\left( x^{1/4+\varepsilon} \right),
\end{equation*}
where
\begin{equation*}
K_j= 2\prod_{k=1}^{\infty} \left(1-\frac1{2^{k/j}} \right)-1 \quad
(j=1,2,3).
\end{equation*}
\end{theorem}

Note that $K_1\doteq -0.422423$, where the digits of
$\prod_{k=1}^{\infty} \left(1-\frac1{2^{k}}\right) \doteq 0.288788$
form the sequence \seqnum{A048651} in the OEIS \cite{OEIS}.

\begin{proof} We use the method described in Section \ref{Subsect_Method}. According to \eqref{sum},
\begin{equation*}
\sum_{n\le x} (-1)^{n-1} a(n) =\sum_{d\le x} h_a(d) \sum_{j\le x/d}
a(j).
\end{equation*}

Remark that by Euler's pentagonal number theorem,
\begin{equation*}
\prod_{k=1}^{\infty} \left(1-\frac1{2^{ks}} \right) =
1+\sum_{j=1}^{\infty} (-1)^j \left(\frac1{2^{(3j-1)js/2}}+
\frac1{2^{(3j+1)js/2}} \right) \quad (|2^s| > 1),
\end{equation*}
which shows that $h_{a}(2^\nu) \in \{-2,0,2 \}$ for every $\nu \ge
1$ and $h_a(p^\nu)=0$ for every prime $p>2$ and every $\nu \ge 1$.
Hence the function $h_a$ is bounded. Now using \eqref{asympt_a_n},
the proof can be carried out by the usual arguments.
\end{proof}

It is known that
\begin{equation*}
\sum_{n\le x} \frac1{a(n)} = D x + O\left(x^{1/2}(\log x)^{-1/2}
\right),
\end{equation*}
where
\begin{equation*}
D= \prod_{p\in \P} \left(1+\sum_{k=2}^{\infty}
\left(\frac1{P(k)}-\frac1{P(k-1)} \right) \frac1{p^k} \right)\doteq
0.752015
\end{equation*}
(sequence \seqnum{A084911} in the OEIS \cite{OEIS}), due to De~Koninck and Ivi\'c
\cite[Thm.\ 1.3]{DeKIvi1980}. Here $P(k)$ denotes the number of
unrestricted partitions of $k$ (not to be confused with the gcd-sum
function from Section \ref{Subsection_gcd}, denoted also by $P$).
See Nowak \cite{Now1991} for a more precise asymptotic formula.

It follows by \eqref{nonzero} that the limit
\begin{equation*}
\lim_{x\to \infty} \frac1{x} \sum_{n\le x} (-1)^{n-1} \frac1{a(n)} =
D \left(2\left(1+\sum_{\nu=1}^{\infty} \frac1{P(\nu)2^{\nu}}
\right)^{-1}-1 \right)
\end{equation*}
exists.

To establish an asymptotic formula for
\begin{equation} \label{asympt_1_per_a}
\sum_{n\le x} (-1)^{n-1} \frac1{a(n)}
\end{equation}
we need to estimate the coefficients of the reciprocal of the power
series $S_{1/a}(x)= 1+ \sum_{\nu=1}^{\infty} \frac1{P(\nu)}x^\nu$.
Here Lemma \ref{lemma_Kaluza} cannot be used, since the sequence
$(a_\nu)_{\nu \ge 0}$ with $a_0=1$ and $a_\nu= \frac1{P(\nu)}$ ($\nu
\ge 1$) is not log-convex. However, observe
that DeSalvo and Pak \cite[Thm.\ 1.1]{DeSPak2015} recently proved
that the sequence $(P(n))$ is log-concave for $n>25$, that is, $(1/P(n))$ is log-convex for
$n>25$.

\begin{open}  Estimate the alternating sum \eqref{asympt_1_per_a}.
\end{open}

\subsection{Sum-of-unitary-divisors function}
\label{Sect_sigma_star}

Recall that $d$ is said to be a unitary divisor of $n$ if $d\mid n$ and $\gcd(d,n/d)=1$.
Let $\sigma^*(n)$ denote, as usual, the sum of
unitary divisors of $n$. The function $\sigma^*$ is multiplicative
and $\sigma^*(p^\nu)=p^{\nu}+1$ ($\nu \ge 1$). One has

\begin{equation*}
\sum_{n=1}^{\infty} \frac{\sigma^*(n)}{n^s}  =
\frac{\zeta(s)\zeta(s-1)}{\zeta(2s-1)} \quad (\Re s>2).
\end{equation*}

Furthermore,
\begin{equation*}
\sum_{n\le x} \sigma^*(n) = \frac{\pi^2}{12\zeta(3)} x^2 + O\left( x
(\log x)^{5/3} \right),
\end{equation*}
established by Sitaramachandrarao and Suryanarayana \cite[Eq.\
(1.4)]{SitSur1973}.

\begin{proposition}
\begin{equation} \label{Dirichlet_sigma_star}
\sum_{n=1}^{\infty} (-1)^{n-1} \frac{\sigma^*(n)}{n^s}  =
\left(1-\frac{6}{2^s}+\frac{6}{2^{2s}}\right)
\left(1-\frac2{2^{2s}}\right)^{-1}
\frac{\zeta(s)\zeta(s-1)}{\zeta(2s-1)}  \quad (\Re s>2).
\end{equation}
\end{proposition}

\begin{theorem}
\begin{equation*}
\sum_{n\le x} (-1)^{n-1} \sigma^*(n) = - \frac{\pi^2}{84\zeta(3)}
x^2 + O\left( x (\log x)^{5/3} \right).
\end{equation*}
\end{theorem}

\begin{proof} Apply Proposition \ref{Prop_f}. The Dirichlet series representation
\eqref{Dirichlet_sigma_star} can be used, the function
$h_{\sigma^*}$ is bounded.
\end{proof}

It is known that
\begin{equation*}
\sum_{n\le x} \frac1{\sigma^*(n)} = B^*\log x+ D^* +O \left(x^{-1}
(\log x)^{5/3} (\log \log x)^{4/3} \right),
\end{equation*}
obtained by Sita~Ramaiah and Suryanarayana \cite[p.\
1352]{SitSur1980}, where $B^*$ and $D^*$ are explicit constants.
Here, according to Proposition \ref{prop_mean_2},
\begin{equation*}
B^*= \prod_{p\in \P} \left(1-\frac1{p} \right)\left(1+
\sum_{\nu=1}^{\infty} \frac1{p^\nu+1} \right).
\end{equation*}

It follows from the same Proposition \ref{prop_mean_2} that 
the limit
\begin{equation*}
E^*:= \lim_{x\to \infty} \frac1{\log x} \sum_{n\le x} (-1)^{n-1}
\frac1{\sigma^*(n)} =  B^* \left(2\left(1+\sum_{\nu=1}^{\infty}
\frac1{2^{\nu}+1} \right)^{-1}-1 \right)
\end{equation*}
exists.

Moreover, by Corollary \ref{Cor_1_per_f} we deduce that
\begin{equation} \label{asympt_1_per_sigma_star}
\sum_{n\le x} (-1)^{n-1} \frac1{\sigma^*(n)} = E^* \log x + F^* + O
\left(x^{-u} (\log x)^{5/3} (\log \log x)^{4/3} \right),
\end{equation}
with an explicit constant $F^*$ and some $u>0$. Bordell\`{e}s and Cloitre \cite[Cor.\ 4, (vi)]{BorClo2013}
established that the error term of \eqref{asympt_1_per_sigma_star}
is $O\left(x^{-1} (\log x)^4\right)$.

To use our method, we need a better estimate for the coefficients
$b_{\nu}$ of the reciprocal of the power series
\begin{equation*}
S_{1/\sigma^*}(x)= 1+\sum_{\nu=1}^{\infty}
\frac{x^{\nu}}{2^{\nu}+1}.
\end{equation*}

Here $b_0=1$, $b_1=-\frac1{3}$, $b_2=-\frac{4}{45}$,
$b_3=-\frac{2}{135}$, $b_4=\frac{32}{34425}$, etc.

\begin{open} We conjecture that $b_{\nu} \ll 1/2^{\nu}$ as $\nu \to \infty$.
If this is true, then it follows from Proposition \ref{Prop_1_per_f}
that the error term in \eqref{asympt_1_per_sigma_star} can be
improved into $O\left(x^{-1} (\log x)^{8/3} (\log \log
x)^{4/3}\right)$.

We pose as an open problem to prove this.
\end{open}

\subsection{Unitary Euler function}

Let $\varphi^*$ be the unitary analog of Euler's $\varphi$ function.
The function $\varphi^*$ is multiplicative and
$\varphi(p^\nu)=p^{\nu}-1$ for every prime power $p^\nu$ ($\nu \ge
1$). One has

\begin{equation*}
\sum_{n=1}^{\infty} \frac{\varphi^*(n)}{n^s}  = \zeta(s)\zeta(s-1)
\prod_p \left(1-\frac2{p^s}+\frac1{p^{2s-1}} \right) \quad (\Re
s>2).
\end{equation*}

Furthermore,
\begin{equation*}
\sum_{n\le x} \varphi^*(n) = \frac{C}{2} x^2 + O\left( x (\log
x)^{5/3}(\log \log x)^{4/3} \right),
\end{equation*}
where $C$ is defined by \eqref{def_C}. See Sitaramachandrarao and
Suryanarayana \cite[Eq.\ (1.5)]{SitSur1973}.

\begin{proposition}
\begin{equation} \label{Dir_phi_star}
\sum_{n=1}^{\infty} (-1)^{n-1} \frac{\varphi^*(n)}{n^s}  =
\left(1-\frac1{2^{s-2}}+\frac1{2^{2s-1}} \right)
\left(1-\frac1{2^{s-1}}+\frac1{2^{2s-1}} \right)^{-1}
\sum_{n=1}^{\infty} \frac{\varphi^*(n)}{n^s}\quad (\Re s>2),
\end{equation}
\end{proposition}

\begin{theorem}
\begin{equation*}
\sum_{n\le x} (-1)^{n-1} \varphi^*(n) = \frac{C}{10} x^2 + O\left( x
(\log x)^{5/3}(\log \log x)^{4/3} \right),
\end{equation*}
where $C$ is defined by \eqref{def_C}.
\end{theorem}

\begin{proof} Apply Proposition \ref{Prop_f}. The Dirichlet series representation \eqref{Dir_phi_star}
can be used. The function $h_{\varphi^*}$ is bounded.
\end{proof}

It is known that
\begin{equation*}
\sum_{n\le x} \frac1{\varphi^*(n)} = L^*\log x+ M^* +O \left(x^{-1}
(\log x)^{5/3}\right),
\end{equation*}
due to Sita~Ramaiah and Subbarao \cite[Thm.\ 3.1]{SitSub1991},
improving the error term of Sita~Ramaiah and Suryanarayana
\cite[Thm.\ 3.2]{SitSur1980}, where $L^*$ and $M^*$ are explicit
constants. Here, according to Proposition \ref{prop_mean_2},
\begin{equation*}
L^*= \prod_{p\in \P} \left(1-\frac1{p} \right)\left(1+
\sum_{\nu=1}^{\infty} \frac1{p^\nu-1} \right).
\end{equation*}

It follows from the same Proposition \ref{prop_mean_2} that 
the limit
\begin{equation*}
T^*:
 = \lim_{x\to \infty} \frac1{\log x} \sum_{n\le x} (-1)^{n-1}
\frac1{\varphi^*(n)} =  L^* \left(\frac{2}{1+K}-1 \right)
\end{equation*}
exists,
where $K$ is the Erd\H{os}-Borwein constant defined by
\eqref{const_Erdos}. Moreover, by Corollary \ref{Cor_1_per_f} (take
$q=1/2$) we deduce that
\begin{equation} \label{asympt_1_per_varphi_star}
\sum_{n\le x} (-1)^{n-1} \frac1{\varphi^*(n)} = T^* \log x + U^* +
O\left(x^{-u} (\log x)^{5/3} \right).
\end{equation}
with an explicit constant $U^*$ and some $u>0$.

Note that this example was not considered by Bordell\`{e}s and
Cloitre \cite{BorClo2013}. To use our method one needs to consider
the power series
\begin{equation*}
S_{1/\varphi^*}(x)= 1+ \sum_{\nu=1}^{\infty}
\frac{x^{\nu}}{2^{\nu}-1},
\end{equation*}
where the sequence $a_0=1$, $a_\nu= \frac1{2^{\nu}-1}$ is log-convex
but only for $\nu \ge 1$, that is $a_\nu^2\le a_{\nu-1}a_{\nu+1}$
holds for $\nu \ge 2$ and is false for $\nu=1$. Hence Lemma
\ref{lemma_Kaluza} cannot be used. In fact, the coefficients
$b_{\nu}$ of the reciprocal power series are $b_0=1$,
$b_1=-1$, $b_2=\frac{2}{3}$, $b_3=-\frac{10}{21}$,
$b_4=\frac{104}{315}$, etc. (not all of $b_1, b_2, \ldots$ are
negative).

\begin{open} We conjecture that $b_{\nu} \ll 1/2^{\nu}$ as $\nu \to \infty$.
If this is true, then it follows from Proposition \ref{Prop_1_per_f}
that the error term in \eqref{asympt_1_per_varphi_star} can be
improved into $O\left(x^{-1} (\log x)^{8/3}\right)$.

We pose as an open problem to prove this.
\end{open}

\subsection{Unitary squarefree kernel} \label{Sect_Unitary squarefree kernel}

Let $\kappa^*(n)$ denote the greatest squarefree unitary divisor of $n$. The function $\kappa^*$ is multiplicative, $\kappa^*(p)=p$ and
$\kappa^*(p^\nu)=1$ for every prime $p$ and $\nu \ge 2$. One has
\begin{equation} \label{kappastar}
\sum_{n\le x} \kappa^*(n)= \frac1{2} \widetilde{C} x^2 + O(R_{\kappa^*}(x)),
\end{equation}
where
\begin{equation} \label{tilde_C}
\widetilde{C} = \prod_{p\in \P}
\left(1-\frac{p^2+p-1}{p^3(p+1)} \right),
\end{equation}
$R_{\kappa^*}(x)= x^{3/2}\delta(x)$, with $\delta(x)$ defined by \eqref{delta}, unconditionally, and $R_{\kappa^*}(x)= x^{7/5}\omega(x)$, with
$\omega(x)$ defined by \eqref{omega}, assuming RH, due to Sita Ramaiah and Suryanarayana \cite[Thm.\ 5.7, 5.8]{SitSur1982}.

\begin{theorem} With the notation above,
\begin{equation} \label{altern_kappastar}
\sum_{n\le x} (-1)^{n-1} \kappa^*(n)= \frac{5}{38} \widetilde{C} x^2
+ O(R_{\kappa^*}(x)).
\end{equation}
\end{theorem}

\begin{proof} We have
\begin{equation*}
S_{\kappa^*}(x)= \frac{x^2-x-1}{x-1} \quad (|x|<1)
\end{equation*}
and the proof is quite similar to the proof of Theorem \ref{Th_kappa}.
\end{proof}

An asymptotic formula for the reciprocal of $\kappa^*(n)$ is simpler
to obtain than for the reciprocal of the squarefree kernel $\kappa(n)$,
discussed in Section \ref{Sect_Squarefree kernel}. It is a result of  Suryanarayana and Subrahmanyam \cite[Cor.\ 3.4.1]{SurSub1981} that
\begin{equation} \label{asymp_1_per_kappa_star}
\sum_{n\le x} \frac1{\kappa^*(n)}= \frac{A\zeta(3/2)}{\zeta(3)} x^{1/2} + \frac{B\zeta(2/3)}{\zeta(2)}x^{1/3}+ O(x^{1/5}),
\end{equation}
where
\begin{equation*}
A=\prod_{p\in \P} \left(1+\frac{\sqrt{p}-1}{p(p-\sqrt{p}+1)} \right), \quad
B=\prod_{p\in \P} \left(1+\frac{p^{1/3}-1}{p(p^{2/3}-p^{1/3}+1)} \right),
\end{equation*}

We deduce the next result.
\begin{theorem}
\begin{equation} \label{altern_1/kappa_star}
\sum_{n\le x} (-1)^{n-1} \frac1{\kappa^*(n)}= \frac{A^*\zeta(3/2)}{\zeta(3)} x^{1/2} + \frac{B^*\zeta(2/3)}{\zeta(2)}x^{1/3}+ O(x^{1/5}),
\end{equation}
where
\begin{equation} \label{Astar_Bstar}
A^*=\frac{A(9-12\sqrt{2})}{23}, \quad
B^*=\frac{B(2^{5/3}-3\cdot 2^{1/3}-1)}{2^{5/3}- 2^{1/3}+1} .
\end{equation}
\end{theorem}

\begin{proof} We have
\begin{equation*}
S_{1/\kappa^*}(x)= \frac{x^2-x+2}{2(1-x)}  \quad (|x|<1),
\end{equation*}
hence
\begin{equation*}
\overline{S}_{1/\kappa^*}(x)= \frac{2(1-x)}{x^2-x+2}=\sum_{\nu=0}^{\infty} b_{\nu}x^\nu  \quad (|x|<\sqrt{2}),
\end{equation*}
where
\begin{equation*}
b_{\nu}= \frac1{4^{\nu+1}\sqrt{7}} \RE \left((\sqrt{7}+i)(1-i\sqrt{7})^{\nu+1}+ (\sqrt{7}-i)(1+i\sqrt{7})^{\nu+1} \right) \quad (\nu \ge 0).
\end{equation*}

Therefore,
\begin{equation*}
|b_{\nu}| \le \frac{4}{\sqrt{7}}\cdot \frac1{2^{\nu/2}} \quad (\nu \ge 1)
\end{equation*}
and using our method this implies \eqref{altern_1/kappa_star}.
\end{proof}

\subsection{Powerful part of a number}

It is possible to deduce similar formulas for many other special multiplicative functions. We consider the following further example.
Every positive integer $n$ can be uniquely written as $n = ab$, where $\gcd(a, b) = 1$, $a$ is squarefree and $b$ is
squareful. Here $b$ is called the {\it powerful part\/} of $n$ and is denoted by $\pow(n)$. See Cloutier, De Koninck, Doyon \cite{CKD2014}. Note that
\begin{equation} \label{kappastar_rel}
\pow(n)=\frac{n}{\kappa^*(n)} \quad (n\ge 1),
\end{equation}
where $\kappa^*(n)$ is the unitary squarefree kernel of $n$, discussed in Section \ref{Sect_Unitary squarefree kernel}.

By partial summation we deduce from \eqref{asymp_1_per_kappa_star} that
\begin{equation} \label{asymp_pow}
\sum_{n\le x} \pow(n)= \frac1{3} c_1 x^{3/2} + \frac1{4} c_2 x^{4/3}+ O(x^{6/5}),
\end{equation}
where
\begin{equation*}
c_1= \prod_{p\in \P} \left(1+\frac{2}{p^{3/2}}-\frac1{p^{5/2}} \right),
\end{equation*}
\begin{equation*}
c_2= \prod_{p\in \P} \left(1+\frac1{p^{2/3}}+\frac{2}{p^{4/3}}-\frac1{p^{7/3}} \right).
\end{equation*}

We remark that \eqref{asymp_pow} is better than \cite[Eq.\ (1)]{CKD2014},
where the error term is $O(x^{4/3})$.

\begin{theorem}
\begin{equation*}
\sum_{n\le x} (-1)^{n-1} \pow(n)= \frac{A^*\zeta(3/2)}{3\zeta(3)} x^{3/2} + \frac{B^*\zeta(4/3)}{4\zeta(2)}x^{4/3}
+ O(x^{6/5}),
\end{equation*}
where the constants $A^*$ and $B^*$ are defined by \eqref{Astar_Bstar}.
\end{theorem}

\begin{proof} Use formulas \eqref{kappastar_rel}, \eqref{altern_1/kappa_star} and partial summation. Alternatively, formula \eqref{asymp_pow}
and the method of the present paper can be applied.
\end{proof}

By partial summation again, we deduce from \eqref{kappastar_rel} and \eqref{kappastar} that
\begin{equation} \label{HH}
\sum_{n\le x} \frac1{\pow(n)}= \widetilde{C} x + O(R_{1/\pow}(x))
\end{equation}
where $\widetilde{C}$ is defined by \eqref{tilde_C}, $R_{1/\pow}(x)= x^{1/2}\delta(x)$, with $\delta(x)$ defined by \eqref{delta},
unconditionally, and $R_{1/\pow}(x)= x^{2/5}\omega(x)$, with $\omega(x)$ defined by \eqref{omega}, assuming RH. Note that this error term is better than $O(x^{1/2})$, indicated in \cite[Eq.\ (3)]{CKD2014}.

\begin{theorem}
\begin{equation*}
\sum_{n\le x} (-1)^{n-1} \frac1{\pow(n)}= \frac{5}{19} \widetilde{C} x + O(R_{1/\pow}(x)),
\end{equation*}
with the notation above.
\end{theorem}

\begin{proof} Apply formulas \eqref{kappastar_rel}, \eqref{altern_kappastar} and partial summation. Alternatively, formula \eqref{HH}
and the method of the present paper can be used.
\end{proof}

\subsection{Sum-of-bi-unitary-divisors function} \label{Sect_sigma**}

Let $\sigma^{**}(n)$ denote, as usual, the sum of bi-unitary divisors of $n$.
Recall that a divisor $d$ of $n$ is a bi-unitary divisor if the greatest common unitary divisor of $d$ and $n/d$ is $1$.
The function $\sigma^{**}$ is multiplicative and for any prime power $p^\nu$ ($\nu \ge 1$),
\begin{equation*}
\sigma^{**}(p^\nu)=\begin{cases} \sigma(p^\nu), & \text{ if $\nu$ is
odd};\\ \sigma(p^\nu)-p^{\nu/2}, & \text{ if $\nu$ is even}.
\end{cases}
\end{equation*}

It is the result of Suryanarayana and Subbarao \cite[Cor. 3.4.3]{SurSub1980} that
\begin{equation*}
\sum_{n\le x} \sigma^{**}(n)= \frac1{2} C^{**} x^2 + O(x(\log x)^3),
\end{equation*}
where
\begin{equation*}
C^{**}= \zeta(2) \zeta(3) \prod_{p\in \P}
\left(1-\frac{2}{p^3}+\frac1{p^4}+\frac1{p^5}-\frac1{p^6} \right).
\end{equation*}

\begin{theorem} We have
\begin{equation*}
\sum_{n\le x} (-1)^{n-1} \sigma^{**}(n)= -\frac{11}{106} C^{**} x^2+ O(x(\log x)^3).
\end{equation*}
\end{theorem}

\begin{proof} Similar to the proof of \eqref{sum_sigma}, by applying Proposition \ref{Prop_f} for $f=\sigma^{**}$.
\end{proof}

Sitaramaiah and Subbarao \cite[Thm.\ 3.2]{SitSub1991} established that
\begin{equation*}
\sum_{n\le x} \frac1{\sigma^{**}(n)}=A^{**}\log x + B^{**} + O(x^{-1}(\log
x)^{14/3}(\log \log x)^{4/3}),
\end{equation*}
where $A^{**}, B^{**}$ are certain explicit constants.

\begin{theorem} We have
\begin{equation} \label{1_per_sigma_star_star}
\sum_{n\le x} (-1)^{n-1}\frac1{\sigma^{**}(n)}=A_1^{**}\log x +B_1^{**}+
O(x^{c}(\log x)^{14/3}(\log \log x)^{4/3}),
\end{equation}
where $A_1^{**}, B_1^{**}$ are explicit constants and $c=(\log 9/10)/(\log 2)\doteq -0.152003$.
\end{theorem}

\begin{proof} Now Lemma \ref{lemma_Kaluza} (theorem of Kaluza) cannot be used, since the sequence
$\left(\frac1{\sigma^{**}(2^{\nu})}\right)_{\nu \ge 0}$ is not log-convex. But it is easy to check that
\begin{equation*}
\frac1{\sigma^{**}(2^\nu)}\le \frac{4}{5}\cdot \frac1{2^\nu} \quad (\nu \ge 1),
\end{equation*}
hence Corollary \ref{Cor_explicit} can be applied with $A=4/5$, $q=1/2$, where $M=q(A+1)=9/10<1$.
\end{proof}

\begin{open}  Improve the error term of \eqref{1_per_sigma_star_star}.
\end{open}

\subsection{Alternating sum-of-divisors function}

Consider the function $\beta(n)=\sum_{d\mid n} d\, \lambda(n/d)$
($n\ge 1$), where $\lambda$ is the Liouville function. The function
$\beta$ is multiplicative and $\beta(p^\nu)= p^\nu - p^{\nu-1}
+p^{\nu-2}-\cdots +(-1)^{\nu}$ for every prime power $p^\nu$ ($\nu
\ge 1$). See the survey paper of the author \cite{Tot2013Sap}.

\begin{proposition}
\begin{equation} \label{D_altern_beta}
\sum_{n=1}^{\infty} (-1)^{n-1} \frac{\beta(n)}{n^s}  =
\left(1-\frac{2}{2^s}-\frac{4}{2^{2s}}\right)
\frac{\zeta(s-1)\zeta(2s)}{\zeta(s)} \quad (\Re s>2).
\end{equation}
\end{proposition}

\begin{proof} Use Proposition
\ref{prop_1} and the representation
\begin{equation*}
\sum_{n=1}^{\infty} \frac{\beta(n)}{n^s}  =
\frac{\zeta(s-1)\zeta(2s)}{\zeta(s)} \quad (\Re s>2).
\end{equation*}
\end{proof}

\begin{theorem}
\begin{equation} \label{sum_beta}
\sum_{n\le x} (-1)^{n-1} \beta(n) = \frac{\pi^2}{120} x^2 + O\left(
x (\log x)^{2/3} (\log \log x)^{4/3} \right).
\end{equation}
\end{theorem}

\begin{proof} Apply Proposition \ref{Prop_f} for $f=\beta$. It is known that
\begin{equation*}
\sum_{n\le x} \beta(n) = \frac{\pi^2}{30} x^2 + O\left( x (\log
x)^{2/3} (\log \log x)^{4/3} \right),
\end{equation*}
see \cite[Eq.\ (15)]{Tot2013Sap}, which is a consequence of the
result of Walfisz \cite[Satz 4, p.\ 144]{Wal1963} for Euler's
$\varphi$ function. The coefficient of the main term in
\eqref{sum_beta} is from \eqref{D_altern_beta},
\begin{equation*}
\frac{\pi^2}{30} \left[1-\frac{2}{2^s}-
\frac{4}{2^{2s}}\right]_{s=2} = \frac{\pi^2}{120}.
\end{equation*}
\end{proof}

It is known (\cite[Eq.\ (17)]{Tot2013Sap}) that for every
$\varepsilon >0$,
\begin{equation} \label{asymp_1_per_beta}
\sum_{n\le x} \frac1{\beta(n)} = K_1\log x+
K_2+O(x^{-1+\varepsilon}),
\end{equation}
where $K_1$ and $K_2$ are constants. Since $\left(\beta(2^\nu)
\right)_{\nu \ge 0}$ is nondecreasing and $\beta(2^\nu) \ge
2^{\nu-1}$  ($\nu \ge 1$), Corollary \ref{Cor_1_per_f} can be
applied (take $q=1/2$). We deduce that
\begin{equation} \label{asymp_1_per_beta_altern}
\sum_{n\le x} (-1)^{n-1} \frac1{\beta(n)} = K_3\log x+ K_4+O(x^{-u})
\end{equation}
with some constants $K_3, K_4$ and some $u>0$.

\begin{open}  Improve the error terms of \eqref{asymp_1_per_beta} and \eqref{asymp_1_per_beta_altern}.
\end{open}

\subsection{Exponential divisor function}

The exponential divisor function $\tau^{(e)}$ is multiplicative and
$\tau^{(e)}(p^\nu)= \tau(\nu)$ for every prime power $p^\nu$ ($\nu
\ge 1$), where $\tau$ is the classical divisor function. There are constants $A_1$ and $A_2$ such that
\begin{equation*}
\sum_{n\le x} \tau^{(e)}(n)  = A_1 x + A_2 x^{1/2} + O(R_{\tau^{(e)}}(x)),
\end{equation*}
where
\begin{equation*}
A_1= \prod_{p\in \P} \left(1+ \sum_{\nu=2}^{\infty} \frac{\tau(\nu)-\tau(\nu-1)}{p^\nu}\right)
\end{equation*}
is the mean value of $\tau^{(e)}$ and $R_{\tau^{(e)}}(x)=x^{2/9}\log
x$, as shown by Wu \cite[Thm.\ 1]{Wu1995}. This error term is strongly
related to estimates on the divisor function $d(1,2;n)=\sum_{ab^2=n} 1$. It can be sharpened into $O(x^{1057/4785 +\varepsilon})$
by using \cite[Thm.\ 1]{GraKol1988}. Also see \cite[Remark, p.\ 135]{Wu1995}.

It follows from Proposition \ref{prop_mean} that the limit
\begin{equation*}
\lim_{x\to \infty} \frac1{x} \sum_{n\le x} (-1)^{n-1} \tau^{(e)}(n) =  A_1 \left(\frac{2}{1+K} -1 \right)
\end{equation*}
exists, where $K=\sum_{\nu=1}^{\infty}
\frac{\tau(\nu)}{2^\nu}=\sum_{\nu=1}^{\infty} \frac1{2^\nu-1}$ is
the Erd\H{o}s-Borwein constant, already quoted above.

\begin{open} Investigate the alternating sums
\begin{equation*}
\sum_{n\le x} (-1)^{n-1} \tau^{(e)}(n), \qquad \sum_{n\le x} (-1)^{n-1} \frac1{\tau^{(e)}(n)}.
\end{equation*}
\end{open}

\section{Generalized alternating sums}

It is possible to investigate the following generalization of the alternating sums discussed above.
Let $Q$ be an arbitrary subset of the set of primes $\P$, let
\begin{equation*}
t_Q(n):= \begin{cases}  1, & \text{if $q\nmid n$ for every $q\in Q$;} \\
-1, & \text{otherwise,}
\end{cases}
\end{equation*}
and
\begin{equation} \label{D_Q}
D_Q(f,s):=\sum_{n=1}^{\infty} t_Q(n) \frac{f(n)}{n^s}.
\end{equation}

If $Q=\{2\}$, then $t_{\{2\}}(n)=(-1)^{n-1}$ and \eqref{D_Q} reduces
to the alternating Dirichlet series \eqref{altern_Dir}.
If $Q=\{2, 3\}$, just to illustrate another special case, then we have
\begin{equation*}
D_{\{2, 3\}}(f,s) = \frac{f(1)}{1^s} -\frac{f(2)}{2^s}
-\frac{f(3)}{3^s} -\frac{f(4)}{4^s} + \frac{f(5)}{5^s} -
\frac{f(6)}{6^s} + \frac{f(7)}{7^s} +\cdots,
\end{equation*}
while the choice $Q=\emptyset$ gives the classical Dirichlet series \eqref{Dir}.

Note that the function $n\mapsto t_Q(n)$ is multiplicative if and only if $Q = \{q\}$ having one element. Proof: If $Q= \{q\}$,
then the function $t_{\{q\}}(n)$ is multiplicative. On the other hand, if there are distinct $q_1,q_2\in Q$, then $t_Q(q_1q_2)=-1 \ne
1 =(-1)(-1)= t_Q(q_1)t_Q(q_2)$. However, the function
\begin{equation*}
c_Q(n):=  \begin{cases}  1, & \text{if $q\nmid n$ for every $q\in Q$;} \\
0, & \text{otherwise;}
\end{cases}
\end{equation*}
is multiplicative for every $Q\subseteq \P$.

\begin{proposition} \label{gen_alt} Let $Q$ be an arbitrary subset of $\P$. If $f$ is a multiplicative function, then
\begin{equation*}
D_Q(f,s) = D(f,s) \left( 2 \prod_{q\in Q} \left(
\sum_{\nu=0}^{\infty}\frac{f(q^\nu)}{q^{\nu s}} \right)^{-1}
-1\right),
\end{equation*}
and if $f$ is completely multiplicative, then
\begin{equation*}
D_Q(f,s) = \prod_{p\in \P} \left(1-\frac{f(p)}{p^s} \right)^{-1} \left( 2 \prod_{q\in Q} \left(1-\frac{f(q)}{q^s} \right)
-1\right),
\end{equation*}
formally or in case of convergence.
\end{proposition}

\begin{proof} We have
\begin{gather*}
D_Q(f,s) = - \sum_{n=1}^{\infty} \frac{f(n)}{n^s} + 2 \sum_{n=1}^{\infty} c_Q(n) \frac{f(n)}{n^s}
 = - D(f,s) +2 \prod_{p\notin Q} \sum_{\nu=0}^{\infty} \frac{f(p^{\nu})}{p^{\nu s}} \\ = - D(f,s) +2
D(f,s) \prod_{q\in Q} \left(\sum_{\nu=0}^{\infty}
\frac{f(q^{\nu})}{q^{\nu s}}\right)^{-1} = D(f,s) \left( 2
\prod_{q\in Q} \left( \sum_{\nu=0}^{\infty}\frac{f(q^\nu)}{q^{\nu
s}} \right)^{-1} -1\right).
\end{gather*}
\end{proof}

If $Q=\{2\}$, then Proposition \ref{gen_alt} reduces to Proposition \ref{prop_1}.

Some of the discussed asymptotic formulas can also be generalized to
certain subsets $Q\subseteq \P$.  For example, we have the next result.

\begin{theorem} Let $Q$ be an arbitrary  finite subset of $\P$. Then
\begin{equation} \label{sum_sigma_Q}
\sum_{n\le x} t_Q(n) \sigma(n)= \frac{\pi^2}{12} \left(2\prod_{p\in Q} \left(1-\frac1{p}\right)
\left(1-\frac1{p^2}\right)-1\right) x^2
+ O\left(x (\log x)^{2/3} \right).
\end{equation}
\end{theorem}

\begin{proof} We have
\begin{equation*}
\sum_{n\le x} t_Q(n) \sigma(n) = - \sum_{n\le x} \sigma(n) + 2 \sum_{n\le x} c_Q(n)\sigma(n),
\end{equation*}
where $c_Q(n)\sigma(n)$ is multiplicative and
\begin{equation*}
\sum_{n=1}^{\infty} \frac{c_Q(n)\sigma(n)}{n^s} = \zeta(s) \zeta(s-1) \prod_{p\in Q}
\left(1-\frac{p+1}{p^s}+\frac{p}{p^{2s}} \right).
\end{equation*}

It turns out that
\begin{equation*}
\sum_{n\le x} c_Q(n)\sigma(n)=\sum_{d\le x} h_Q(d) \sum_{j\le x/d} \sigma(j),
\end{equation*}
where the function $h_Q$ is multiplicative and for every prime power $p^\nu$ ($\nu \ge 1$),
\begin{equation*}
h_Q(p^\nu)= \begin{cases} -(p+1), & \text{if $p\in Q$, $\nu=1$;} \\
                               p, & \text{if $p\in Q$, $\nu=2$;} \\
                               0, & \text{otherwise.} \end{cases}
\end{equation*}

Now the proof runs similar to the proof of \eqref{sum_sigma}.
\end{proof}

In the case $Q=\{2\}$ formula \eqref{sum_sigma_Q} reduces to \eqref{sum_sigma}.

\begin{open} Deduce asymptotic formulas for
\begin{equation*}
\sum_{n\le x} t_Q(n) \sigma(n)
\end{equation*}
and for similar sums if $Q$ is an arbitrary fixed subset of the primes.
\end{open}

\section{Acknowledgments} The author thanks Shyamal Biswas for raising the
question of investigating alternating Dirichlet series, Roberto
Tauraso for drawing attention to the paper of Kaluza \cite{Kal1928},
and G\'erald Tenenbaum for helpful remarks on Section
\ref{Sect_Squarefree kernel}. The author is grateful to the
anonymous referee for identity \eqref{b_nu}, Proposition
\ref{Prop_expl} and many other useful comments and
suggestions.

\medskip

(Concerned with the following sequences in \cite{OEIS}: \seqnum{A000005}, \seqnum{A000010}, \seqnum{A000041}, \seqnum{A000203}, \seqnum{A000688},
\seqnum{A001615}, \seqnum{A002088}, \seqnum{A005117}, \seqnum{A006218}, \seqnum{A007947}, \seqnum{A013928}, \seqnum{A018804},
\seqnum{A024916}, \seqnum{A033999}, \seqnum{A034448}, \seqnum{A047994}, \seqnum{A048651}, \seqnum{A049419}, \seqnum{A057521},
\seqnum{A063966}, \seqnum{A064609}, \seqnum{A065442}, \seqnum{A065463}, \seqnum{A068762}, \seqnum{A068773}, \seqnum{A084911},
\seqnum{A143348}, \seqnum{A145353}, \seqnum{A173290}, \seqnum{A177754}, \seqnum{A188999}, \seqnum{A206369}, \seqnum{A272718}.)

\end{document}